\documentclass[a4paper,reqno]{amsart}

\usepackage{amsmath}
\usepackage{paralist}
\usepackage[colorlinks=true]{hyperref}
\hypersetup{urlcolor=blue, citecolor=red}
\usepackage{amsfonts}
\usepackage{enumerate}
\usepackage{amssymb}

\numberwithin{equation}{section}

\newtheorem{theorem}{Theorem}[section]
\newtheorem{lemma}[theorem]{Lemma}
\newtheorem{proposition}[theorem]{Proposition}
\theoremstyle{definition}
\newtheorem{remark}[theorem]{Remark}

\newenvironment{notation}{\medskip \noindent{\bf Notation. }}{}

\newcommand\R{{\mathbb R}}
\newcommand\C{{\mathbb C}}
\newcommand\N{{\mathbb N}}

\newcommand\dist{{\mathrm{d}}}
\newcommand\Srn{{\mathcal S}(\R^N )}
\newcommand\Sseul{{\mathcal S}}
\newcommand\Fou{{\mathcal F}}
\newcommand\Fmu{{\mathcal F}^{-1}}
\newcommand\Loc{{\mathrm{loc}}}

\newcommand\MScN[1]{\href{http://www.ams.org/mathscinet-getitem?mr=#1}{\nolinkurl{(#1)}}}
\newcommand\DOII[1]{\href{http://dx.doi.org/#1}{(doi: \nolinkurl{#1})}}
\newcommand\LINK[1]{\href{#1}{(link: \nolinkurl{#1})}}

\newcommand\NL{{\mathcal N}}
\newcommand\LI{{\mathcal M}}
\newcommand\LINL{{\mathcal A}}
\newcommand\Spa{{\mathcal X}}
\newcommand\Spb{{\mathcal Y}}
\newcommand\Ens{{\mathcal E}}
\newcommand\SpaD{{\widetilde{{\mathcal X}}}}
\newcommand\SpbD{{\widetilde{{\mathcal Y}}}}

\newcommand\mns{-}
\newcommand\Cstu{\mu _1}
\newcommand\Cstd{\mu _2}
\newcommand\Cstt{\mu _3}

\newcommand\CDiracu{\ell}
\newcommand\CDiracd{{2^\ell}}

\title[Local smooth solutions]{Local smooth solutions of the nonlinear Klein-gordon equation}

\author[Thierry Cazenave and Ivan Naumkin]{}



\email{\href{mailto:thierry.cazenave@sorbonne-universite.fr}{thierry.cazenave@sorbonne-universite.fr}}
\email{\href{mailto:ivan.naumkin@iimas.unam.mx}{ivan.naumkin@iimas.unam.mx}}

\thanks{Ivan Naumkin is a Fellow of Sistema Nacional de Investigadores. The research was partially supported by project PAPIIT IA101820}

\begin{document}

\maketitle

\centerline{\scshape Thierry Cazenave}
\medskip
{\footnotesize
 \centerline{Sorbonne Universit\'e, CNRS, Universit\'e de Paris} 
  \centerline{Laboratoire Jacques-Louis Lions, B.C. 187}
   \centerline{4 place Jussieu, 75252 Paris Cedex 05, France}}
\medskip
\centerline{\scshape  Ivan Naumkin}
{\footnotesize \centerline{Departamento de F\'{\i}sica Matem\'{a}tica}
 \centerline{Instituto de Investigaciones
en Matem\'{a}ticas Aplicadas y en Sistemas}
 \centerline{Universidad Nacional Aut\'{o}noma
de M\'{e}xico}
   \centerline{Apartado Postal 20-126, Ciudad de M\'{e}xico, 01000, M\'{e}xico}
}

\bigskip

\begin{abstract}
Given any $\mu _1, \mu _2\in {\mathbb C}$ and $\alpha >0$, we prove the local existence of arbitrarily smooth solutions of the nonlinear Klein-Gordon equation $\partial  _{ tt } u - \Delta u + \mu _1 u = \mu _2  |u|^\alpha u$ on ${\mathbb R}^N$, $N\ge 1$, that do not vanish, i.e. $ |u (t,x) | >0 $ for all $x \in {\mathbb R}^N$ and all sufficiently small $t$.
We write the equation in the form of a first-order system associated with a pseudo-differential operator, then use a method adapted from~[Commun. Contemp. Math.  {\bf 19} (2017), no. 2, 1650038]. 
We also apply a similar (but simpler than in the case of the Klein-Gordon equation) argument to prove an analogous result for a class of nonlinear Dirac equations.
\end{abstract}


\section{Introduction}

We study the local existence of smooth solutions for the nonlinear Klein-Gordon equation on $ \R^N $, $N\ge 1$, 
\begin{equation} \label{KGCauchy}
\begin{cases} 
w_{tt} - \Delta  w +  \Cstu w = \Cstd  \vert w \vert ^{\alpha}w, \\
w\left(  0,x\right)  =w_{0}\left(  x\right)  \text{, }w_{t}\left(  0,x\right)
=w_{1}(x),
\end{cases} 
\end{equation}
 where $\alpha>0$ and $ \Cstu ,  \Cstd \in \C $. Note that if $\Cstu =0$, then~\eqref{KGCauchy} is in fact the nonlinear wave equation. 
 
Local well-posedness of the Cauchy problem~\eqref{KGCauchy} in the energy space $H^1 (\R^N ) \times L^2 (\R^N ) $
is established in~\cite{GV85, GV89} in the subcritical case $ (N-2) \alpha <4$.
Local existence in $H^s (\R^N ) \times H^{s-1} (\R^N ) $ with $0\le s< \frac {N} {2}$ is established in~\cite{pecherW}, provided $( N- 2s )\alpha \le 4$, and also that the nonlinearity is sufficiently smooth. Typically, it is required that $\alpha >[s-1]$. 
There are many more references on this topic but, as far as we are aware, they do not cover the case of all powers $\alpha >0$ in all spatial dimensions $N\ge 1$.
Under appropriate assumptions on $\Cstu, \Cstd$ and $\alpha $, it is known that the solutions of~\eqref{KGCauchy} are global and scatter as $  |t| \to \infty $ (i.e., they behave like solutions to the linear equation) for small initial values (low energy scattering), or for all initial values (asymptotic completeness). 
For asymptotic completeness, see for instance~\cite{MorawetzS, Brenner1, Brenner2, GV85-2, GV89-2, Nakanishi}.
There is a considerable literature on the subject, of which we mention only a small fraction. 
For low energy scattering, see for instance~\cite{Strauss1, Strauss2, Pecher1, Pecher2, Klainerman85, shatah, ozawa,Moriyama97, Mori-Tone-Tsutsu97,Katayama99,delortfang2004}.
In particular, one expects low energy scattering when $\alpha >\frac {2} {N}$. In the case $\alpha =\frac {2} {N}$ one expects low energy modified scattering, i.e. that small solutions of~\eqref{KGCauchy} behave as linear
solutions modulated by a phase. 
This is known in space dimensions $N=1, 2$. See~\cite{HaNa3} for the case $\alpha >\frac {2} {N}$, and~\cite{Delort01,LindbladSoffer1, HaNa, Masaki} for the case $\alpha =\frac {2} {N}$. (See also
\cite{NaumkinKGD,NaumkinKGN} for the corresponding initial-boundary value problems.) 
In higher space dimensions $N\geq3$, up to our knowledge,  the best available result is the existence of scattering for power nonlinearities $\alpha\geq\alpha_{0}$ for some $\alpha_{0}>2/N$. Therefore, there is a gap
between the expectation and what is actually known.

In this paper, we construct for every $N\ge 1$ and $\alpha>0$ a class of initial values for
which there exists a local, \textit{highly-regular}, \textit{non-vanishing} solution of~\eqref{KGCauchy}. The difficulty is that, since $\alpha>0$ can be small with respect to the required regularity, the nonlinearity is not smooth enough. See \cite{CDWreg, CazNau} for a discussion on this regularity issue.
In~\cite{CazNau,CazNau1} we proved similar results for the Schr\"o\-din\-ger equation with nonlinearity $ |u|^\alpha u$. 
These results were useful, via the pseudo-conformal transformation, to study the scattering problem for NLS with $\alpha\geq2/N$ close to the critical power $\alpha=2/N$. The highly-regular
solutions were also used to prove local existence for the generalized
derivative Schr\"{o}dinger equation \cite{Linares,Linares1} and to the
generalized Korteweg-de Vries equation \cite{Linares2}.
We expect that the results in the present paper will be useful to derive similar results for the nonlinear
Klein-Gordon equation~\eqref{KGCauchy}.

Before stating our
results, we introduce some notation taken from~\cite{CazNau}. We fix $\alpha>0$, we consider three
integers ${k},{m},{n}$ such that
\begin{equation}
{k}>\frac{N}{2},\quad{n}>\max\left\{  \frac{N}{2}+1,\frac{N}{2\alpha}\right\}
,\quad2{m}\geq{k}+{n}+3 \label{fDInt1}%
\end{equation}
and we let
\begin{equation}
{J}=2{m}+2+{k}. \label{fSpa1b1}%
\end{equation}
Let $d\geq1.$ We define the space $ \Spa _{d}$ by
\begin{align}
 \Spa _{d}  &  =\left\{  u\in H^{{J}}( \R^N , \C
^{d});\,\langle x\rangle^{{n}}D^{\beta}u\in L^{\infty}( \R
^{N}, \C ^{d})\text{ for }0\leq|\beta|\leq2{m-2;}\right. \nonumber\\
&  \left.  \langle x\rangle^{{n}}D^{\beta}u\in L^{2}({ \R }%
^{N} ,  \C ^{d})\text{ for }2{m}-1 \leq\left\vert \beta\right\vert \leq
J\right\}  . \label{fSpa1}%
\end{align}
and we equip $ \Spa _{d}$ with the norm
\begin{equation}
\Vert u\Vert_{ \Spa _{d}}=\sum_{{  |\beta | }=0}^{2{m-2}} 
\Vert\langle x\rangle^{{n}}D^{\beta}u\Vert_{L^{\infty}({ \R }%
^{N},  \C ^{d})}+\sum_{|\beta|=2{m}-1} ^J \Vert\langle x\rangle^{{n}}D^{\beta
}u\Vert_{L^{2}( \R^N ,  \C ^{d})} \label{fSpa2}%
\end{equation}
where
\begin{equation*} 
\langle x\rangle=(1+|x|^{2})^{\frac{1}{2}}.
\end{equation*} 
By standard considerations it follows that $( \Spa _{d}
,\Vert\cdot\Vert_{ \Spa _{d}})$ is a Banach space.
Moreover, $2{n}>N+2$ by~\eqref{fDInt1}, so that $\Vert\langle x\rangle w\Vert_{L^{2}}\leq C\Vert\langle x\rangle^{n} w\Vert_{L^{\infty}}$; and so
\begin{equation} \label{fSpa1:b2}
 \Spa _{d}\hookrightarrow H^{J}( \R^N ,  \C ^{d}).
\end{equation}
It is straightforward to check that $ \Sseul ({ \R }
^{N}, \C ^{d})\subset \Spa _{d}$,  and that  $z\langle x\rangle^{-p}
\in \Spa _{d}$ for all $p\geq{n}$ and $z\in \C ^{d}.$

Our main result for equation~\eqref{KGCauchy} is the following.

\begin{theorem}
\label{eThm1b1} 
Let $\alpha>0$ and $  \Cstu , \Cstd  \in{ \C }$. 
Assume~\eqref{fDInt1}-\eqref{fSpa1b1} and let $ \Spa _{1}$ be defined
by~\eqref{fSpa1}-\eqref{fSpa2} with $d=1$. 
Let $w_0, w_1\in H^1 (\R^N ) \times L^2 (\R^N ) $
satisfy $w_{0}\in \Spa _{1}$, $\left\langle i\nabla\right\rangle
^{-1}w_{1}\in \Spa _{1}$, where $  \langle i\nabla \rangle = (I- \Delta )^{\frac {1} {2}} = \Fmu \left\langle
\xi\right\rangle \Fou $. If
\begin{equation} \label{BoundBelow}
\inf_{x\in \R^N }\langle x\rangle^{n} (   \vert w_{0} (
x )   \vert + \vert  \langle i\nabla \rangle ^{-1}
w_{1} (  x )   \vert  )  >0, 
\end{equation}
then there exist $T>0$ and a unique solution  $w\in C([-T,T], \Spa _{1})$ of~\eqref{KGCauchy}. 
In addition, if 
\begin{equation}
\inf_{x\in \R^N }\langle x\rangle^{n}\left\vert w_{0}\left(
x\right)  \right\vert >0, \label{14}%
\end{equation}
then there exists $0<T_{1}<T $ such that the solution $w (  t )  $ does not
vanish for all $ \vert t \vert \leq T_{1}$, more precisely there exists $\eta>0$ such
that
\begin{equation} \label{15} 
\inf  _{ -T_1 \le t \le T_1 } \inf_{x\in \R^N }\langle x\rangle^{n} \vert w (  t,x )
 \vert \geq\eta .
\end{equation}
\end{theorem}

\begin{remark}
\label{eRem2}
Here are some comments on Theorem~$\ref{eThm1b1}$.
\begin{enumerate}[{\rm (i)}] 

\item \label{eRem2:2} The parameters ${k},{m},{n}$ are arbitrary as long as they
satisfy~\eqref{fDInt1}. In particular, ${n}$ can be any integer satisfying the
second condition in~\eqref{fDInt1}.

\item \label{eRem2:5} There are no restrictions on the size of the initial
value in Theorem~\ref{eThm1b1}. Besides the smoothness and decay imposed by
the assumption $w_{0} ,  \langle i\nabla \rangle ^{-1}w_{1}\in{\mathcal{X}} _{1}$, the only limitation is condition~\eqref{BoundBelow}.

\item Theorem~{\ref{eThm1b1}} applies to the initial data
\begin{equation*} 
w_{0} (x) =z_{0}(\langle x\rangle^{-{n}}+\psi_{0} (x) ), \quad w_{1} (x) =z_{1} \langle
i\nabla \rangle (\langle x\rangle^{-{n}}+\psi_{1}) (x) ,
\end{equation*} 
where $z_{j}\in \C $, $j=1,2$,  ${n}>\max\{\frac{N}{2}+1,\frac
{N}{2\alpha}\}$, and $\psi_{j}\in \Srn $ satisfies
$\| \langle \cdot \rangle^{{n}}\psi_{j} \|_{L^{\infty}}<1$, for $j=1,2$.

\item We cannot guarantee that the solution $w (  t )  $ does not
vanish if condition \eqref{14} is not fulfilled. 
We do not know if condition \eqref{14} is necessary. 
\end{enumerate}
\end{remark}

\textit{Comments on the proof of Theorem~$\ref{eThm1b1}$}. We do not work
directly with equation \eqref{KGCauchy}. Instead, we reduce it to a first
order in time system and study the resulting ``half-wave equation". Namely,
given $w_{0},w_{1}$ let
\begin{equation*} 
u_{0}=2^{-1} ( w_0  \mathbf{a} +i [ \langle i\nabla
\rangle ^{-1}w_{1} ] \mathbf{b})  ,
\end{equation*} 
where
\begin{equation} \label{fDfnab} 
\mathbf{a} = 
\begin{pmatrix} 
1 \\1
\end{pmatrix} 
\quad 
\mathbf{b} = 
\begin{pmatrix} 
1 \\ -1
\end{pmatrix} 
\end{equation} 
We consider the first order system
\begin{equation} \label{Halfwave}
\begin{cases} 
i \partial_{t}u-\gamma\left\langle i\nabla\right\rangle u= \LI (u) +  \NL(
u)  ,\\
u (  0 )  =u_{0},
\end{cases} 
\end{equation}
where  $\gamma$
is the Pauli matrix
\begin{equation} \label{gamma}
\gamma= 
\begin{pmatrix}
1 & 0\\
0 & -1
\end{pmatrix} , 
\end{equation}
and 
\begin{equation} \label{fdefnnl} 
\LI (u) =  \frac { \Cstu -1 } {2} [ \langle i\nabla \rangle ^{-1} (   \mathbf{a} \cdot u  ) ] \mathbf{b}, \quad \NL (u) = - \frac { \Cstd  } {2} [ \langle i\nabla \rangle ^{-1} ( | \mathbf{a} \cdot u |^\alpha  \mathbf{a} \cdot u  ) ] \mathbf{b} .
\end{equation} 
In particular, $\gamma \langle i\nabla \rangle $ with domain $H^1 (\R^N , \C^2 ) $ is self-adjoint  on $L^2 (\R^N , \C^2 ) $,
and $(e^{it\gamma \langle i\nabla \rangle })  _{ t\in \R }$ is a group of isometries on $L^2 (\R^N , \C^2 ) $. Moreover, since $\gamma \langle i\nabla \rangle$ commutes with any power of $(I - \Delta )$, $(e^{it\gamma \langle i\nabla \rangle })  _{ t\in \R }$ is also a group of isometries on  $H^s (\R^N , \C^2 ) $ for all $s\in \R$.
We can rewrite \eqref{KGCauchy} as \eqref{Halfwave} by letting (see~\cite{HaNa1})
\begin{equation}  \label{transform}
u=\frac{1}{2} ( w  \mathbf{a} + i [ \langle i\nabla \rangle
^{-1}\partial_{t}w]  \mathbf{b}  )  .
\end{equation}
Moreover, $w$ is given in terms of $u$ by
\begin{equation*} 
w=\mathbf{a} \cdot u .
\end{equation*} 
Therefore, we concentrate our attention on the problem
\eqref{Halfwave}. Similarly to the non-relativistic case of the
Schr\"{o}dinger equation~\cite{CazNau}, we observe that the possible defect of
smoothness of the nonlinearity $ \NL \left(  u\right)  $ is only at
$u=0$. This observation suggests to look for a solution to \eqref{Halfwave}
that does not vanish. We follow the approach of~\cite{CazNau} to construct such
solutions. In order to explain our strategy, we consider as initial data the
case of the concrete function $\psi(x)= \langle x\rangle^{-{n}} \mathbf{a}$,
where ${n}>\frac{N}{2}+1.$ For this choice of $n$ it follows that $\psi\in
H^{1}( \R^N  ,  \C ^{2})$. Let $v(t)=e^{it\gamma \langle
i\nabla \rangle }\psi$ be the solution of the linear problem
\begin{equation} \label{Low2}
\begin{cases}
iv_{t}=\gamma\left\langle i\nabla\right\rangle v,\\
v(0,x)=\psi(x).
\end{cases}
\end{equation}
We want to control $\inf_{x\in \R^N }\langle x\rangle^{n}|v(t,x)|$.
For this purpose we integrate \eqref{Low2} on $[0,t]$, and we obtain
\begin{equation}
v(t,x)=\psi(x)-i\gamma\int_{0}^{t}\left(  \left\langle i\nabla\right\rangle
v\right)  (s,x)\,ds. \label{fInt1}%
\end{equation}
Hence,
\begin{equation}
\inf_{x\in \R^N }\langle x\rangle^{n}|v(t,x)|\geq\inf_{x\in
 \R^N }\langle x\rangle^{n}|\psi(x)|-t\Vert\langle x\rangle
^{n}\left\langle i\nabla\right\rangle v\Vert_{L^{\infty}((0,t)\times
 \R^N )}. \label{Low4}
\end{equation}
It follows that in order to control $\inf_{x\in \R^N }\langle
x\rangle^{n}|v(t,x)|$ from below, we need to estimate the last term on the
right-hand side of~\eqref{Low4}. In \cite{CazNau} we use Taylor's formula with
integral remainder involving derivatives of the solution $v$ of sufficiently
large order, which we estimate in the Sobolev space $H^{s}$ where $s>\frac
{N}{2}$ and $k$ is sufficiently large. In the case of equation \eqref{fInt1}
we have
\begin{equation}
\Vert\langle x\rangle^{n}\left\langle i\nabla\right\rangle v\Vert_{L^{\infty}%
}=\Vert\langle x\rangle^{n}\left\langle i\nabla\right\rangle ^{-1}\left(
1- \Delta \right)  v\Vert_{L^{\infty}}. \label{19}%
\end{equation}
The Taylor's formula applied to $v(t)=e^{it\gamma\left\langle i\nabla
\right\rangle }\psi$ yields terms of the form $\left\langle i\nabla
\right\rangle ^{{j}}\psi$ and a remainder involving the operator $\left(
1- \Delta \right)  ^{{\ell }},$ for some $\ell $ large. To estimate the terms
$\left\langle i\nabla\right\rangle ^{{j}}\psi,$ we need to control the
pseudo-differential operator $\left\langle i\nabla\right\rangle ^{{j}}$. This
is done in Lemma \ref{BE} below by using the theory of Bessel potentials.
In turn, expanding $\left(  1- \Delta \right)  ^{{\ell }}$ to extract the part
corresponding to $v$ itself, for small times we can estimate the solution in
terms of its higher-order derivatives. Hence, we can control $v$ in terms of
derivatives of $\psi$ plus a high-order derivative of $v,$ $\sup_{\left\vert
\beta\right\vert = \ell }\Vert\langle x\rangle^{n}D^{\beta}v\Vert_{L^{\infty}},$
for some $ \ell $ large, which is estimated via the Sobolev's embedding
$H^{s}\hookrightarrow L^{\infty}$ for $s>\frac{N}{2}.$ This first step is
achieved in Lemma~\ref{eLE1} below. In order to control $\sup_{\left\vert
\beta\right\vert  = \ell }\Vert\langle x\rangle^{n}D^{\beta}v\Vert_{H^{s}}$ we use
energy estimates. Since the equation for $v$ involves a first order
pseudo-differential operator, we can estimate $\langle x\rangle^{n}D^{\beta}v$
in terms of $D^{\beta}v,$ which then can be controlled by usual energy
estimates. This second step is achieved in Lemma~\ref{eLE3:0} below. As in the
non-relativistic case of \cite{CazNau}, the corresponding space  $ \Spa _2$ involves weighted
$L^{\infty}$-norms of the derivatives of the function up to a certain order,
then weighted $L^{2}$-norms of the derivatives of higher order.\ However, we
stress that in the present case we are able to close the estimates using lower
order derivatives, compared with the case of the Schr\"{o}dinger equation.
Combining Lemmas~\ref{eLE1} and~\ref{eLE3:0} we obtain the linear estimate
that we use in this paper. This estimate is presented in
Proposition~\ref{eLE3} below. In particular, it follows from the linear
estimate that if  $\langle x\rangle^{n}|\psi |$ is bounded from below, then the function
$\langle x\rangle^{n}|v(t,x)|$ remains positive for small times. The nonlinear
estimate is provided by Proposition~\ref{eNL1} below. We show that we can
close the nonlinear estimates in the space $ \Spa _2$ via the control of
the Bessel potential provided by Lemma \ref{BE}. Using the linear and
nonlinear estimates, we prove the existence of a solution $u\in \Spa _2$
for \eqref{Halfwave} by a contraction mapping argument. This is the result of
Proposition~\ref{eNLS1} below. Then, Theorem~\ref{eThm1b1} follows from the
transformation \eqref{transform}, which relates the problems \eqref{KGCauchy} and \eqref{Halfwave}.

It turns out that the method we use to study~\eqref{KGCauchy} similarly applies to the nonlinear Dirac equation%
\begin{equation}  \label{Dirac}
\begin{cases} 
i\Psi_{t} = H\Psi+ \Cstt   | \Psi | ^{\alpha}  \Psi,\\
\Psi (  0,x )  =\Psi_{0} (  x ),
\end{cases} 
\end{equation}
on $\R^N $. Here, $ \Cstt \in \C $, $\Psi (
t,x )  \in\C ^{ \CDiracd }$, where 
\begin{equation} \label{fDfnell} 
 \CDiracu = \Bigl[ \frac{N+1}{2} \Bigr] , 
\end{equation} 
 with $[x]$  the
integer part of $x$, and  $\langle \cdot.\cdot \rangle $ denotes
 the scalar product in $ \C ^{n}$. The $N$-dimensional free
Dirac operator $H$ is defined by
\begin{equation} \label{HDirac}
H= - i \sum_{k=1}^{N} \gamma_{k} \partial_{k}+ \eta, 
\end{equation}
where the $ \CDiracd \times \CDiracd $ Hermitian matrices $\gamma_{k},\eta$ satisfy the
anticommutation relations 
\begin{equation}  \label{anti-commutation}
\begin{cases} 
\gamma_{j}\gamma_{k}+\gamma_{k}\gamma_{j}=2\delta_{jk} I , & j,k\in\{1,2,...,N\}, \\
\gamma_{j}\eta+\eta\gamma_{j}=0 ,\\
\eta^{2}=I .
\end{cases} 
\end{equation} 
If 
\begin{equation*} 
\sigma_{1}=%
\begin{pmatrix}
0 & 1\\
1 & 0
\end{pmatrix}
,\quad \sigma_{2}=%
\begin{pmatrix}
0 & -i\\
i & 0
\end{pmatrix}
,\quad \sigma_{3}=%
\begin{pmatrix}
1 & 0\\
0 & -1
\end{pmatrix} ;
\end{equation*} 
are the Pauli matrices, then a standard choice in dimension one is $\gamma _1= \sigma _1$, $\eta=\sigma_{3}$.
In the case of dimension two, the usual convention is $ \gamma_{1}=\sigma_{1}$, $\gamma_{2}=\sigma_{2}$,  $\eta=\sigma_{3}$. 
 In dimension three, the standard convention is
\begin{equation*} 
\gamma_{j}=%
\begin{pmatrix}
0 & \sigma_{j}\\
\sigma_{j} & 0
\end{pmatrix}
,\text{ \ }1\leq j\leq3, \quad \eta=%
\begin{pmatrix}
I  & 0\\
0 & -I %
\end{pmatrix}
,
\end{equation*} 
where in the above formula $I$ is the $2\times 2$ identity matrix.
In higher dimensions $N\geq4,$ the full set of explicit matrices $\gamma
_{j},\eta$ satisfying the anti-commutation relations~\eqref{anti-commutation} 
can be constructed by iteration (see e.g. the Appendix in~\cite{KalfY}). 

The Dirac equation with power nonlinearity like~\eqref{Dirac} was studied in particular in~\cite{DiasF, Escveg, Tzvetkov, Machihara1, Machihara2}. However,  to the best of our knowledge, and as for equation~\eqref{KGCauchy}, the available local well-posedness results do not cover the case of all powers $\alpha >0$ in all spatial dimensions $N\ge 1$.
Our main result for equation~\eqref{Dirac} is the following.

\begin{theorem} \label{ThDirac}
Let $\alpha>0$, $ \Cstt \in \C $, and let  $ \CDiracu $ be given by~\eqref{fDfnell}.
Assume~\eqref{fDInt1}-\eqref{fSpa1b1} and let $ \Spa _{\CDiracd}$ be defined
by~\eqref{fSpa1}-\eqref{fSpa2} with $d=\CDiracd$. 
Assume $\Psi_{0}\in \Spa  _{ \CDiracd } $ satisfies
\begin{equation*} 
\inf_{x\in \R^N }\langle x\rangle^{n}|\Psi_{0}(x)|>0.
\end{equation*} 
It follows that there exist $T>0$ and a unique solution $\Psi\in
C([-T,T], \Spa  _{ \CDiracd }) $ of~\eqref{Dirac}. In addition, there exists $\eta >0$ such that
\begin{equation*} 
\inf  _{ -T \le t\le T }\inf_{x\in \R^N }\langle x\rangle^{n} \vert \Psi (
t,x )   \vert \geq\eta .
\end{equation*} 
\end{theorem}

\begin{remark} 
The Dirac equation is often studied with the nonlinearity $\langle \eta \Psi ,\Psi \rangle \eta \Psi $, in this case it is known as the Thirring model \cite{Thirring} (see \cite{Soler} for three space dimensions). 
Many results are available for the Thirring model, see for instance~\cite{MNO, MNNO, MNT, SelbergT, Candy1,pecherD, BejenaruHerr, Sasaki,BejenaruHerr1,Candy} (see also \cite{NaumkinDirac1,NaumkinDirac2}).
This type of nonlinearity is not accessible to the method we use in this paper. Indeed, we would need lower estimates of $ | \langle \eta \Psi ,\Psi \rangle |$, which do not follow immediately from the method we use to prove Proposition~\ref{eLE3Dirac}. 
\end{remark} 

\bigskip\textit{Comments on the proof of Theorem~$\ref{ThDirac}$.}
 The strategy of the proof is the same as in the case of the problem
\eqref{Halfwave} above. Indeed, let $v(t)=e^{itH}\psi$ be the solution of the
linear problem
\begin{equation*} 
\begin{cases}
iv_{t}=Hv,\\
v(0,x)=\psi(x).
\end{cases}
\end{equation*} 
Integrating, we have
\begin{equation*} 
v(t,x)=\psi(x)-\int_{0}^{t} (  Hv )  (s,x)\,ds.
\end{equation*} 
Hence,
\begin{equation*} 
\inf_{x \in  \R^N }\langle x\rangle^{n}|v(t,x)|\geq \inf_{x\in
 \R^N } \langle x\rangle^{n}|\psi(x)|-t\Vert\langle x\rangle
^{n}Hv\Vert_{L^{\infty}((0,t)\times \R^N )}. 
\end{equation*} 
As $H$ is a first-order differential operator, we are in a similar situation
as for (\ref{Halfwave}). In fact, the case of the Dirac equation results
somehow easier to handle, since there is no pseudo-differential operator involved.
See Proposition~\ref{eLE3Dirac} below. 

As mentioned earlier, a natural application of the above results would be to address the existence of
scattering or modified scattering for the nonlinear Klein-Gordon~\eqref{KGCauchy} and Dirac~\eqref{Dirac}  equations
 in any spatial dimensions $N\geq1,$ similar to the results
obtained in \cite{CazNau, CazNau1} for the case of the Schr\"{o}dinger
equation. 
This question is more challenging in these cases
because the corresponding conformal-type transforms involve 
pseudo-differential operators in the case of the Klein-Gordon equations, and
a system of equations in the case of the Dirac equation. Therefore, 
the action on the nonlinearity of these conformal-type
transforms results to be difficult to control. 

The rest of this paper is organized as follows. In Section~\ref{sLow} we
establish estimates for the group $(e^{it\gamma \langle
i\nabla \rangle })_{t\in \R }$ in the space $ \Spa _{2}$ (Proposition~\ref{eLE3}), and of the group $(e^{itH})_{t\in \R} $ in the space $\Spa _{ \CDiracd }$ (Proposition~\ref{eLE3Dirac}). 
In Section~\ref{sNL} we estimate the nonlinearity $ \NL (  u )  $ in $ \Spa _{2}$ and $ \NL _{1} (
\Psi )  = \Cstt  | \Psi | ^{\alpha}  \Psi$ in $\Spa _{\CDiracd }$. Finally, in Section~\ref{sNLS} we complete the proofs of
Theorems~\ref{eThm1b1} and~\ref{ThDirac}.

\begin{notation}
We denote by $L^{p}(U ,  \C ^d )$, for $1\leq p\leq\infty$ and
$U= \R^N $ or $U=(0,T)\times \R^N $, $0<T\leq\infty$, the
usual  $\C^d$-valued Lebesgue spaces. We use
the standard notation that $\Vert u\Vert_{L^{p}}=\infty$ if $u\in
L ^{1} _\Loc (U ,  \C ^{d})$ and $u\not \in L^{p}(U ,  \C %
^{d})$. $H^{s}( \R^N ,  \C ^{d})$, $s\in{ \R }$, is the
usual $\C^d$-valued Sobolev space. (See
e.g.~\cite{AdamsF} for the definitions and properties of these spaces.) 
We will often write $L^p (U)$ and $H^s (\R^N ) $ for $L^p (U, \C^d)$ and $H^s (\R^N , \C^d) $, respectively.
We denote by $(e^{ \mns it\gamma \langle i\nabla \rangle })_{t\in{ \R }%
}$ the group  associated to the equation~\eqref{Low2}.  
As is well known, $(e^{ - it\gamma\left\langle i\nabla\right\rangle }%
)_{t\in{ \R }}$ is a group of isometries on $L^{2}({ \R }%
^{N} ,  \C ^{2})$, and on $H^{s}( \R^N  ,  \C ^{2})$ for all
$s\in{ \R }$.
\end{notation}

\section{Weighted estimates for the linear equations} \label{sLow}

We first estimate the action of the group $(e^{ \mns it\gamma\left\langle
i\nabla\right\rangle })_{t\in{ \R }}$ on the space 
\begin{equation} \label{fDfnSpa} 
 \Spa= \Spa _{2} . 
\end{equation} 
We prove the following result.

\begin{proposition} \label{eLE3} 
Assume \eqref{fDInt1}-\eqref{fSpa1b1} with $\alpha=1$, and let
the space $ \Spa $ be defined by \eqref{fSpa1}-\eqref{fSpa2} and~\eqref{fDfnSpa}. It follows that $e^{ \mns it\gamma\left\langle i\nabla\right\rangle
}\psi\in C({ \R }, \Spa )$ for all $\psi \in \Spa$. Moreover, there exist 
$C>0$  and $t_{0}>0$ such that
\begin{equation} \label{eLE3:11} 
\Vert e^{ \mns it\gamma\left\langle i\nabla\right\rangle }\psi\Vert_{ \Spa  
}\leq C(1+|t|)^{2m +n +1} \Vert\psi\Vert_{ \Spa } , 
\end{equation}
and
\begin{equation} \label{eLE3:12}
\sup_{|\beta|\leq2{m}}\Vert\langle x\rangle^{n}D^{\beta}(e^{ \mns it\gamma
\left\langle i\nabla\right\rangle }\psi-\psi)\Vert_{L^{\infty}}\leq
C  | t | (1+  | t | )^ {2m +n +1} \Vert\psi\Vert_{ \Spa } ,
\end{equation}
for all $\left\vert t\right\vert \leq t_{0}$ and all $\psi\in \Spa $.
\end{proposition}

Before proving Proposition~\ref{eLE3}, we first establish a weighted
$L^{\infty}$ estimate. Before doing this, we prepare two estimates. First, we
recall an interpolation estimate (see Lemma A.2 of \cite{CazNau}):

\begin{lemma}
Given $j\in \N $ and $\nu \in{ \R }$, there exists a constant $C$
such that
\begin{equation}
\sup_{|\beta|=j+1}\Vert\langle x\rangle^{\nu}D^{\beta}u\Vert_{L^{\infty}}\leq
C(\sup_{|\beta|=j}\Vert\langle x\rangle^{\nu}D^{\beta}u\Vert_{L^{\infty}}%
+\sup_{|\beta|=j+2}\Vert\langle x\rangle^{\nu}D^{\beta}u\Vert_{L^{\infty}})
\label{fSob6}%
\end{equation}
for all $u\in C^{j+2}( \R^N )$.
\end{lemma}

Also, we need the following control of the Bessel potential $\left\langle
i\nabla\right\rangle ^{-1}.$

\begin{lemma}
\label{BE}Let $n\in \N $ and $1\leq p\leq\infty.$ For $f\in L^{p}$ the
estimate
\begin{equation}
\Vert\langle x\rangle^{n}\left\langle i\nabla\right\rangle ^{-1}f\Vert_{L^{p}%
}\leq C\Vert\langle x\rangle^{n}f\Vert_{L^{p}} \label{Besselestimate}%
\end{equation}
is true for some $C>0.$
\end{lemma}

\begin{proof}
We use the theory of Bessel potentials applied to $\left\langle i\nabla
\right\rangle ^{-1}.$ We have (see relation (26), Chapter V of \cite{Stein})%
\[
\left\langle i\nabla\right\rangle ^{-1}f=G\ast f=\int_{ \R ^{N}}G\left(
x-y\right)  f\left(  y\right)  dy,
\]
where%
\begin{equation}
G\left(  x\right)  =\frac{1}{2\pi}\int_{0}^{\infty}\frac{e^{-\frac
{\pi\left\vert x\right\vert ^{2}}{\theta}-\frac{\theta}{4\pi}}}{\theta
^{1+\frac{N-1}{2}}}d\theta. \label{G}%
\end{equation}
Since $\langle x\rangle^{n}\leq C\left(  1+\left\vert x\right\vert
^{n}\right)  ,$ we estimate%
\begin{equation}
\left\vert \langle x\rangle^{n}\left(  \left\langle i\nabla\right\rangle
^{-1}f\right)  \left(  x\right)  \right\vert \leq C\left(  \left\vert \left(
G\ast f\right)  \left(  x\right)  \right\vert +\left\vert \left\vert
x\right\vert ^{n}\left(  G\ast f\right)  \left(  x\right)  \right\vert
\right)  , \label{16}%
\end{equation}
for $x\in \R ^{N}.$ We take into account the estimate for the Bessel
potentials (see relation (36), Chapter V of \cite{Stein})
\begin{equation}
\left\Vert G\ast f\right\Vert _{L^{p}}\leq\left\Vert f\right\Vert _{L^{p}%
},\text{ for }1\leq p\leq\infty. \label{17}%
\end{equation}
Using this estimate, we control the first term in the right-hand side of
\eqref{16} by $\left\Vert f\right\Vert _{L^{p}}.$ Now we use that%
\[
\left\vert x\right\vert ^{n}\leq\left(  \left\vert x-y\right\vert +\left\vert
y\right\vert \right)  ^{n}=\sum_{j=0}^{n}C_{j}^{{n}}\left\vert x-y\right\vert
^{n-j}\left\vert y\right\vert ^{j},
\]
where $C_{j}^{{n}}$ are the binomial coefficients. Then%
\[
\left\vert \left\vert x\right\vert ^{n}\left(  G\ast f\right)  \left(
x\right)  \right\vert \leq C\sum_{j=0}^{n}\int_{ \R ^{N}}\left\vert
x-y\right\vert ^{n-j}G\left(  x-y\right)  \left\vert y\right\vert
^{j}\left\vert f\left(  y\right)  \right\vert dy.
\]
Since $G$ given by \eqref{G} has the singularity $|x|^{-n+1/2}$ at $x=0$ and
decays exponentially as $|x|\rightarrow\infty$ (see \cite[Chapter~V, formulas (29)-(30)]{Stein}), we see that $\left\vert \cdot\right\vert
^{m}G\in L^{1}$, for all $m\geq0$. Hence, by Young inequality we show that%
\begin{equation}
\left\Vert \left\vert \cdot\right\vert ^{n}\left(  G\ast f\right)  \left(
\cdot\right)  \right\Vert _{L^{p}}\leq C\Vert\langle\cdot\rangle^{n}%
f\Vert_{L^{p}},\text{ for }1\leq p\leq\infty. \label{18}%
\end{equation}
Using \eqref{17} and \eqref{18} in \eqref{16} we obtain~\eqref{Besselestimate}.
\end{proof}

We now are in position to prove the weighted $L^{\infty}$ estimate for the
linear flow. We have the following result.

\begin{lemma} \label{eLE1} 
Assume~\eqref{fDInt1}-\eqref{fSpa1b1} with $\alpha=1$. There
exist $C>0$  and $t_{0}>0$  such that
\begin{equation} \label{Low5b1}
\begin{split}
\sum_{{j}=0}^{2 m - 2 }  \sup_{|\beta|={j}} & \Vert\langle x\rangle^{{n}}D^{\beta
}e^{ \mns is\gamma\left\langle i\nabla\right\rangle }\psi\Vert_{L^{\infty
}((0,t)\times \R^N )} \\
\le & C(1+t)^{2 m }\sum_{  |j| \le 2m} \Vert\langle
x\rangle^{{n}}D^{\beta}\psi\Vert_{L^{\infty} } \\ &
+Ct(1+t)^{2m }\sup_{|\beta|=2{m}}\Vert\langle x\rangle^{{n}}D^{\beta
}e^{ \mns is\gamma\left\langle i\nabla\right\rangle }\psi\Vert_{L^{\infty
}((0,t)\times \R^N )},
\end{split}
\end{equation}
for all $0\leq t\leq t_{0}$ and all $\psi\in H^{{J}}( \R^N )$.
\end{lemma}

\begin{proof}
Set $v(t)=e^{ \mns it\gamma\left\langle i\nabla\right\rangle }\psi$. Since
$\left\langle i\nabla\right\rangle ^{{j}}\psi\in H^{{J}-{j}}({ \R }
^{N})$ for $0\leq{j}\leq2{m}$, we have $v\in C^{{j}}([0,\infty),H^{{J}-{j}
}( \R^N ))$ and $\frac{d^{{j}}v}{dt^{{j}}}= ( i \gamma) ^{{j}%
}\left\langle i\nabla\right\rangle ^{j}v(t)$ for all $0\leq{j}\leq2{m}$. Given
$0\leq{\ell}\leq{m -1 }$, we apply Taylor's formula with integral remainder
involving the derivative of order $2\left(  {m}-{\ell}\right)  $ to the
function $v$, and we obtain
\[%
\begin{split}
v(t)=  &  \sum_{{j}=0}^{2{m}-2{\ell}-1}\frac{(it \gamma)^{{j}}}{{j}%
!}\left\langle i\nabla\right\rangle ^{{j}}\psi\\
&  +\frac{\left(  i \gamma\right)  ^{2{m}-2{\ell}} }{(2{m}-2{\ell-1})!}%
\int_{0}^{t}(t-s)^{2{m}-2{\ell}-1}\left(  1- \Delta \right)  ^{{m}-{\ell}%
}v(s)\,ds
\end{split}
\]
for all $t\geq0$. Applying now $D^{\beta}$ with $|\beta|= 2{\ell} $, we deduce
that
\[%
\begin{split}
D^{\beta}v(t)=  &  \sum_{{j}=0}^{2{m}-2{\ell}-1}\frac{(it \gamma)^{{j}}}{{j}!}
D^{\beta}\left\langle i\nabla\right\rangle ^{{j}}\psi\\
&  +\frac{\left(  i \gamma\right)  ^{2{m}-2{\ell}} }{(2{m}-2{\ell-1})!}%
\int_{0}^{t}(t-s)^{2{m}-2{\ell}-1} D^{\beta}\left(  1- \Delta \right)
^{{m}-{\ell}}v(s)\,ds .
\end{split}
\]
Developing the binomial $\left(  1- \Delta \right)  ^{{m}-{\ell}}$, we
obtain
\begin{equation}
\label{eAbs1:2}%
\begin{split}
D^{\beta}v(t)=  &  \sum_{{j}=0}^{2{m}-2{\ell}-1} \frac{(it \gamma)^{{j}}}%
{{j}!} D^{\beta}\langle i \nabla\rangle^{{j}} \psi\\
&  +\frac{\left(  i \gamma\right)  ^{2{m}-2{\ell}} }{(2{m}-2{\ell-1})!} \sum_{
j=0 } ^{m - \ell} (-1)^{j} C_{j}^{m-\ell} \int_{0}^{t}(t-s)^{2{m}-2{\ell}-1}
D^{\beta}\Delta^{j} v(s)\,ds .
\end{split}
\end{equation}
Identity~\eqref{eAbs1:2} holds in $C([0,\infty),H^{{k}}( \R^N ))$,
hence in $C([0,\infty)\times \R^N )$ by Sobolev's embedding. We
multiply~\eqref{eAbs1:2} by $\langle x\rangle^{n} $ and take the supremum in
$x $, then in $t$, to obtain
\begin{equation} 
\label{kg1}
\begin{split}
\Vert\langle x\rangle^{{n}}D^{\beta} v\Vert_{L^{\infty}((0,t)\times
 \R^N )} \le &  C(1+t)^{2{m}} \sum_{{j}=0}^{{2m - 2\ell-1 }}
\Vert\langle x\rangle^{{n}} \langle i\nabla\rangle^{ j}D^{\beta}\psi
\Vert_{L^{\infty} }\\
&  +Ct(1+t)^{2{m}}\sum_{j={0}} ^{{m - \ell}} \Vert\langle x\rangle^{{n}%
}D^{\beta} \Delta^{j} v \Vert_{L^{\infty}( (0,t)\times \R^N )}.
\end{split}
\end{equation} 
We note that if $j$ is even, then $\langle i \nabla\rangle^{{\ j }} = (1-
\Delta)^{\frac{j} {2}}$ and that if $j$ is odd, then $\langle i \nabla
\rangle^{j} = \langle i \nabla\rangle^{-1} (1- \Delta)^{\frac{j+ 1} {2}}$.
Therefore, for $0\le j\le2m- 2\ell-1$, and since $|\beta|= 2{\ell} $, we have
(using~\eqref{Besselestimate} if $j$ is odd)
\begin{equation*} 
\Vert\langle x\rangle^{{n}} \langle i\nabla\rangle^{ j}D^{\beta}\psi
\Vert_{L^{\infty}  } \le\sum_{ |\rho|\le2m } \Vert\langle
x\rangle^{{n}}D^{\rho} \psi\Vert_{L^{\infty} } .
\end{equation*} 
It follows that
\begin{equation} \label{kg2}
\begin{split}
\sup_{ |\rho| =2 \ell} \Vert\langle x\rangle^{{n}}D^{\rho} v\Vert_{L^{\infty
}((0,t)  \times \R^N )} \le  & C(1+t)^{2{m}}\sum_{ |\rho|\le2m }
\Vert\langle x\rangle^{{n}}D^{\rho} \psi\Vert_{L^{\infty} }\\
  +Ct(1+t)^{2{m}} & \sum_{j={\ell}}^{{m}} \sup_{| \rho|=2j}\Vert\langle
x\rangle^{{n}}D^{ \rho} v \Vert_{L^{\infty}((0,t)\times \R^N )}.
\end{split}
\end{equation} 

We now fix $t_{0}>0$ sufficiently small so that
\begin{equation}
\label{fCndTz}Ct_{0} (1+t_{0} )^{2{m}}\le\frac{1} {2} ,
\end{equation}
and we set
\begin{equation}
\label{fDefnA}A= C(1+t_{0})^{2{m}}\sum_{ |\rho|\le2m } \Vert\langle
x\rangle^{{n}}D^{\rho} \psi\Vert_{L^{\infty} } .
\end{equation}
Moreover, for $0<t\le t_{0}$, we set
\begin{equation}
\label{fDefnB}B(t)= C t (1+t)^{2m} \sup_{|\rho|=2m}\Vert\langle x\rangle^{{n}%
}D^{\rho} v \Vert_{L^{\infty}((0,t)\times \R^N )} .
\end{equation}
With this notation, we see that for $0<t\le t_{0}$,
\[%
\begin{split}
\sup_{|\rho|=2{\ell}}\Vert\langle x\rangle^{{n}}D^{\rho}  &  v\Vert
_{L^{\infty}((0,t)\times \R^N )} \le A\\
&  + C t (1+t)^{2{m}} \sum_{j={\ell+1 }}^{{m}}\sup_{|\rho|=2j}\Vert\langle
x\rangle^{{n}}D^{\rho} v \Vert_{L^{\infty}((0,t)\times \R^N )}\\
&  + \frac{1} {2} \sup_{|\rho|=2\ell}\Vert\langle x\rangle^{{n}}D^{\rho} v
\Vert_{L^{\infty}((0,t)\times \R^N )},
\end{split}
\]
so that
\begin{equation}
\label{fEstGen1}%
\begin{split}
\sup_{|\rho|=2{\ell}}\Vert\langle x\rangle^{{n}}D^{\rho}  &  v\Vert
_{L^{\infty}((0,t)\times \R^N )} \le2 A\\
&  + 2 C t (1+t)^{2{m}} \sum_{j={\ell+1 }}^{{m}}\sup_{|\rho|=2j}\Vert\langle
x\rangle^{{n}}D^{\rho} v \Vert_{L^{\infty}((0,t)\times \R^N )} .
\end{split}
\end{equation}
We first apply~\eqref{fEstGen1} with $\ell=m-1$, and we obtain
using~\eqref{fDefnB}
\begin{equation}
\label{fEstGen2}\sup_{|\rho|=2( m-1) }\Vert\langle x\rangle^{{n}}D^{\rho}
v\Vert_{L^{\infty}((0,t)\times \R^N )} \le2 A + 2 B(t) .
\end{equation}
Next, we apply~\eqref{fEstGen1} with $\ell=m- 2$, together
with~\eqref{fEstGen2} and~\eqref{fCndTz}, to obtain
\[
\sup_{|\rho|=2 (m-2) }\Vert\langle x\rangle^{{n}}D^{\rho} v\Vert_{L^{\infty
}((0,t)\times \R^N )} \le4 A + 4 B(t) .
\]
An obvious iteration shows that
\[
\sup_{|\rho|=2 (m- j ) }\Vert\langle x\rangle^{{n}}D^{\rho} v\Vert_{L^{\infty
}((0,t)\times \R^N )} \le2^{j } A + 2^{j} B(t)
\]
for all $1\le j \le m$. Thus we see that
\begin{equation}
\label{fEstGen5}\sup_{ 0 \le\ell\le m-1 }\sup_{|\rho|=2 \ell}\Vert\langle
x\rangle^{{n}}D^{\rho} v\Vert_{L^{\infty}((0,t)\times \R^N )}
\le2^{m} A + 2^{m} B(t) .
\end{equation}
The derivatives of even order in the left-hand side of \eqref{Low5b1} are
estimated by~\eqref{fEstGen5}. Finally, we use the interpolation estimate
\eqref{fSob6} to control the derivatives of odd order in the left-hand side of
\eqref{Low5b1} by the derivatives of even order. This completes the proof of~\eqref{Low5b1}.
\end{proof}

Estimate \eqref{Low5b1} shows that we can control the linear solution in terms
of the initial data and a high-order derivative of this solution. Since we can
estimate the $L^{2}$ norm from the equation via energy estimates, we now
control the uniform norm of the term in the right-hand side of \eqref{Low5b1}
which involves the linear flow by Sobolev's embedding. The last is done by
establishing an appropriate weighted $L^{2}$ estimate (Lemma~\ref{eLE3:0}
below). We first introduce some notation.
Assume~\eqref{fDInt1}-\eqref{fSpa1b1} with $\alpha=1$. We define the space%

\begin{equation} \label{fDfnSpbD1} 
 \Spb =\{u\in H^{J}( \R^N ,  \C ^{2});\,\langle
 x\rangle^{{n}}D^{\beta}u\in L^{2}( \R^N ,  \C ^{2})\text{ for
}2{m}-1 \leq\left\vert \beta\right\vert \leq J\}.
\end{equation} 
We equip $ \Spb $ with the norm
\begin{equation} \label{fDfnSpbD2} 
\Vert u\Vert_{ \Spb }=\sum_{{  |\beta | }=0}^{2{m-2}} \Vert
D^{\beta}u\Vert_{L^{2}}+ \sum_{|\beta|=2{m} -1 }^{J}\Vert\langle x\rangle^{{n}%
}D^{\beta}u\Vert_{L^{2}}.
\end{equation} 
Note that, using~\eqref{fSpa1:b2}, 
\begin{equation} \label{fInjYX} 
 \Spa \hookrightarrow \Spb . 
\end{equation} 
We
observe that $( \Spb ,\Vert\cdot\Vert_{ \Spb })$ is a Banach
space and that $\Sseul ( \R^N , \C^{2})$ is dense in
$ \Spb $. We will use the following commutation relation.

\begin{lemma}
Given any integer $\ell \ge 1$, there exist an integer $ \nu \ge 1$ and functions 
\begin{equation*} 
(a_j) _{ 1\le j\le  \nu   }, (b_j) _{ 1\le j\le  \nu   }, (c_j) _{ 1\le j\le  \nu   } \subset C^\infty (\R^N )
\end{equation*} 
satisfying
\begin{equation} \label{abc}
 | a_j (x) |\le C \langle x\rangle ^{ \ell -1}, \quad  | b_j (\xi ) |\le C, \quad 
 | c_j (x) |\le C \langle x\rangle ^{ \ell }, 
\end{equation}
for $1\le j\le  \nu   $, 
such that
\begin{equation} \label{commutator}
 \langle \cdot \rangle ^{2 \ell } \langle i\nabla \rangle f = \langle i\nabla \rangle (  \langle \cdot    \rangle^{2 \ell } f  ) + \sum_{ j=1 }^ \nu  a_j \Fmu [ b_j \Fou ( c_j f ) ],
\end{equation} 
for all $x\in \R ^{N}$ and $f\in \Srn$.
\end{lemma}

\begin{proof}
Since $\left\langle x\right\rangle ^{2 \ell }=\sum_{j=0}^ \ell C_{j}^ \ell  \left\vert
x\right\vert ^{2j},$ we have
\begin{equation*} 
\begin{split} 
\Fou [ \langle \cdot \rangle ^{2 \ell } \langle i\nabla \rangle f ] & = \sum_{ j=0 }^ \ell   C^ \ell _j \Fou [  | \cdot | ^{2j} \langle i\nabla \rangle f ] = \sum_{ j=0 }^ \ell   C^ \ell _j (- \Delta  )^j \Fou [  \langle i\nabla \rangle f ] \\& = \sum_{ j=0 }^ \ell   C^ \ell _j (- \Delta  )^j [  \langle \xi   \rangle  \widehat {f} (\xi )  ] .
\end{split} 
\end{equation*} 
We observe that 
\begin{equation*} 
(-\Delta )^j ( u v )=  u  [ (-\Delta )^j v] + \sum_{\substack{  |\beta _1| +  |\beta _2| =2j\\  |\beta _1|\ge 1}} \gamma  _{ \beta _1, \beta _2 } [ D^{ \beta _1} u ]  [ D^{ \beta _2} v ]  
\end{equation*} 
for some coefficients $\gamma _{ j_1, j_2 }$.
Therefore, we may write
\begin{equation}\label{24}
\Fou [ \langle \cdot \rangle ^{2 \ell } \langle i\nabla \rangle f ] =I_{1}(  \xi )  +I_{2} ( \xi)  ,
\end{equation}
where
\begin{equation*}
I_{1} (  \xi )  = \sum_{ j=0 }^ \ell   C^ \ell _j  \langle \xi   \rangle [ (- \Delta  )^j   \widehat {f} ]
\end{equation*}
and
\begin{equation*} 
I_2 (\xi ) =  \sum_{\substack{  |\beta _1| +  |\beta _2| \le 2 \ell \\  |\beta _1| \ge 1}} C  _{ \beta _1, \beta _2 } [D^{\beta _1} \langle \xi \rangle  ]  [ D^{ \beta _2}  \widehat{f}  ]
\end{equation*} 
for some coefficients $C_{j_{1},j_{2}}$. 
We have
\begin{equation} \label{21}
I_{1} (  \xi )  = \sum_{ j=0 }^ \ell C^ \ell _j  \langle \xi   \rangle \Fou [  | x |^{2j} f ]
= \langle \xi   \rangle \Fou [  \langle x   \rangle^{2 \ell } f ]
= \Fou [ \langle i\nabla \rangle (  \langle x   \rangle^{2 \ell } f  ) ].
\end{equation}
Taking the
inverse Fourier transform in the expression for $I_{2}\left(  \xi\right)  $ we obtain
\begin{equation} \label{fOM1} 
\Fmu I_2    =  \sum_{\substack{  |\beta _1| +  |\beta _2| \le 2 \ell \\  |\beta _1| \ge 1}} C  _{ \beta _1, \beta _2 } w _{ \beta _1, \beta _2 }, 
\end{equation} 
where
\begin{equation} \label{20}
 w _{ \beta _1, \beta _2 } =  \Fmu [ (D^{\beta _1} \langle \xi \rangle  ) ( D^{ \beta _2}  \widehat{f} ) ] 
 = (2\pi )^{-N}   \int  _{ \R^N  } e^{i x\cdot \xi } (D^{\beta _1} \langle \xi \rangle ) ( D^{ \beta _2}  \widehat{f} )\, d\xi  .
\end{equation} 
We recall that 
\begin{equation} \label{fOM2} 
 | D^\beta \langle \xi \rangle  | \le C 
\end{equation} 
for all $ |\beta |\ge 1$. (See e.g.~\cite[formula~(A.2)]{CazNau1}.)
If $ |\beta _2|\le  \ell $, we write $ w _{ \beta _1, \beta _2 } = a \Fmu [ b \Fou (c f) ]$ with $a(x) \equiv 1$, $b(\xi ) \equiv  D^{\beta _1} \langle \xi \rangle$ and $c(x)= x^{\beta _2}$. Using~\eqref{fOM2}, we see that this is a term allowed by~\eqref{abc}. 
If $ |\beta _2|\ge  \ell +1$, we write $\beta _2= \beta _2' + \beta _2 ''$ with $ |\beta _2 '| = \ell $ and $  |\beta _2'' |\le  \ell -1$ (recall that $ |\beta _2|\le 2 \ell -1$ in the sum~\eqref{fOM1}).
After integration by parts, we obtain
\begin{equation*} 
\begin{split} 
 w _{ \beta _1, \beta _2 } & = (-1)^{ |\beta _2 ''|} (2\pi )^{-N}   \int  _{ \R^N  } D^{\beta _2''} [e^{i x\cdot \xi } (D^{\beta _1} \langle \xi \rangle )]  ( D^{ \beta _2'}  \widehat{f} )\, d\xi \\
  & = (-1)^{ |\beta _2 ''|} (2\pi )^{-N}  \sum_{  \beta _3+\beta _4= \beta _2 '' }  i ^{ |\beta _3|} x^{\beta _3} \int  _{ \R^N  } e^{i x\cdot \xi } (D^{\beta _1+ \beta _4} \langle \xi \rangle ) ( D^{ \beta _2'}  \widehat{f} )\, d\xi .
\end{split} 
\end{equation*} 
Using again~\eqref{fOM2}, we see that each of the terms in the above series has the appropriate form. 
The result now follows from~\eqref{24}, \eqref{21}  and~\eqref{fOM1}. 
\end{proof}

We now prove the following:

\begin{lemma} \label{eLE3:0}
 Assume~\eqref{fDInt1}-\eqref{fSpa1b1} with $\alpha=1$. 
 It follows that, given any $\psi\in \Spb $, $ e^{ \mns it\gamma\left\langle i\nabla
\right\rangle }\psi\in C([0,\infty), \Spb )$. Moreover, there exists a
constant $C$ such that
\begin{equation}  \label{eLE3:11:0}
\Vert e^{ \mns it\gamma \langle i\nabla \rangle }\psi\Vert_{ \Spb }\leq C(1+t)^{{n}}\Vert\psi \Vert_{ \Spb }
\end{equation}
for all $t\geq0$ and all $\psi\in \Spb $.
\end{lemma}

\begin{proof}
We first prove estimate~\eqref{eLE3:11:0} for $\psi\in \Srn $. Let $\psi\in  \Srn $ and
set $v(t)= e^{ \mns it\gamma\left\langle i\nabla\right\rangle }\psi$.
It follows by standard Fourier analysis that $v\in C^\infty  ([0,\infty), \Srn )$. Since the linear flow is
isometric on $H^{2{m-2}}( \R^N )$, we need only estimate the
weighted terms $\Vert \langle  x \rangle  ^{{n}}D^{\beta}v\Vert_{L^{2}},$
with $2{m} -1 \leq\left\vert \beta\right\vert \leq J$. We fix $2{m} -1 \leq\left\vert
\beta\right\vert \leq J$ and we prove that
\begin{equation}
\Vert\langle x\rangle^{{n}}D^{\beta}v\Vert_{L^{2}}\leq C\left(  1+t\right)
^{n}\sum_{  |\rho | =2{m} -1 }^{J} \Vert\langle x\rangle^{{n}}D^{\rho}%
\psi\Vert_{L^{2}}. \label{4}%
\end{equation}
Let $\ell \in \{ 1, \cdots, n \}$.
Applying $D^{\beta}$ to equation~\eqref{Low2}, multiplying by $\left\langle
x\right\rangle ^{2{ \ell }}D^{\beta}\overline{v}$, integrating on ${ \R }%
^{N}$, and taking the imaginary part, we obtain
\begin{equation} \label{3}
\frac{1}{2} \frac {d} {dt} \Vert\langle x\rangle^ \ell D^{\beta} v  \Vert_{L^{2}}
^{2 } = \Im \int_{ \R^N } (\langle x\rangle^{2 \ell }\gamma\langle
i\nabla\rangle D^{\beta}v)\cdot D^{\beta} \overline{v} . 
\end{equation}
We use now the commutation relation~\eqref{commutator} in~\eqref{3}. We have
\begin{equation}  \label{25}
\begin{split} 
\frac{1}{2} \frac {d} {dt} \Vert\langle x\rangle^ \ell D^{\beta}v\Vert_{L^{2}}^{2} = & 
\Im\int_{ \R^N } [  \gamma \langle i\nabla \rangle
 (   \langle \cdot \rangle ^{2 \ell }D^{\beta}v )  ]  \cdot
D^{\beta}\overline{v} \\ & 
 + \sum_{j=1}^{\nu  }\Im\int_{ \R^N } (\gamma  a_{j}   b_{j} (  \nabla )   [
c_{j} D^{\beta}v   ]    )\cdot D^{\beta}\overline{v}.
\end{split} 
\end{equation} 
Since $\gamma \langle i\nabla \rangle $ is self-adjoint, by using again (\ref{commutator}) we have
\begin{equation} \label{fOM3} 
\begin{split} 
  \Im\int_{ \R^N } & [  \gamma \langle i\nabla \rangle
 (   \langle \cdot \rangle ^{2 \ell }D^{\beta}v )  ]  \cdot
D^{\beta}\overline{v} \\
& = \frac {1} {2i}  \Bigl(    \int_{ \R^N } [  \gamma \langle i\nabla \rangle
 (   \langle \cdot \rangle ^{2 \ell }D^{\beta}v )  ]  \cdot
D^{\beta}\overline{v} -   \overline{ \int_{ \R^N }    \langle \cdot \rangle ^{2 \ell }D^{\beta}v   \cdot
[   \gamma \langle i\nabla \rangle D^{\beta}\overline{v} ] } \Bigr) \\
& = \frac {1} {2i}  \Bigl(    \int_{ \R^N } [  \gamma \langle i\nabla \rangle
 (   \langle \cdot \rangle ^{2 \ell }D^{\beta}v )  ]  \cdot
D^{\beta}\overline{v} -   \int_{ \R^N }    [   \gamma  \langle \cdot \rangle ^{2 \ell } \langle i\nabla \rangle D^{\beta} v ]  \cdot D^{\beta}  \overline{v} \Bigr) \\ 
& = - \frac {1} {2i } \sum_{j=1}^{m  }\Im\int_{ \R^N } (\gamma  a_{j}   b_{j} (  \nabla )   [
c_{j} D^{\beta}v   ]    )\cdot D^{\beta}\overline{v}.
\end{split} 
\end{equation} 
It follows from~\eqref{25}, \eqref{fOM3}, \eqref{abc}, and $ \ell \le n$, that
\begin{equation} \label{kg3}
\begin{split} 
\frac{1}{2} \frac {d} {dt} \Vert\langle x\rangle^ \ell D^{\beta}v\Vert_{L^{2}}^{2}  & \le 
C \sum_{j=1}^{\nu  } \int_{ \R^N } \langle x\rangle ^{ \ell -1} | b_{j} (  \nabla )   [
c_{j} D^{\beta}v   ] | \, |  D^{\beta}\overline{v} | \\ & \le 
C \sum_{j=1}^{\nu  }    \| b_{j} (  \nabla )   [
c_{j} D^{\beta}v   ] \| _{ L^2 }  \| \langle \cdot \rangle ^{ \ell -1} D^\beta v\| _{ L^2 }
 \\ & \le 
C \|  \langle \cdot \rangle ^ \ell D^{\beta}v  \| _{ L^2 }  \| \langle \cdot \rangle ^{ \ell -1} D^\beta v\| _{ L^2 } .
\end{split} 
\end{equation} 
Hence 
\begin{equation*} 
\begin{split} 
\Vert\langle x\rangle^ \ell D^{\beta}v (t) \Vert_{L^{2}} & \le  \Vert\langle x\rangle^ \ell D^{\beta} \psi \Vert_{L^{2}} + C t 
\sup _{ 0<\tau <t } \Vert\langle x\rangle^{ \ell -1} D^{\beta}v ( \tau ) \Vert_{L^{2}} \\ & \le  \Vert\langle x\rangle^ n D^{\beta} \psi \Vert_{L^{2}} + C t  \sup _{ 0<\tau <t } \Vert\langle x\rangle^{ \ell -1} D^{\beta}v ( \tau ) \Vert_{L^{2}} .
\end{split} 
\end{equation*} 
We apply successively the above estimate with $\ell =1, \cdots, n$. 
Since $  \Vert   D^{\beta}v ( \tau ) \Vert_{L^{2}} =  \Vert   D^{\beta} \psi  \Vert_{L^{2}} $, we conclude that
\begin{equation*} 
\Vert\langle x\rangle^ n D^{\beta}v (t) \Vert_{L^{2}} \le C (1+t)^n \Vert\langle x\rangle^ n D^{\beta} \psi \Vert_{L^{2}} .
\end{equation*} 
This last estimate proves~\eqref{4}. 

Let now $\psi \in \Spb$ and $(\psi _n) _{ n\ge 1 }\subset \Srn$ such that $\psi _n \to \psi $ in $\Spb$ as $n\to \infty $. 
Applying\eqref{4} with $\psi $ replaced by $\psi _m-\psi _n$, we deduce that for every $T>0$, $ e^{ \mns it\gamma \langle i\nabla \rangle }\psi _n$ is a Cauchy sequence in $L^\infty ((-T, T), \Spb )$. It follows that $ e^{ \mns it\gamma \langle i\nabla \rangle }\psi  $ belongs to $C([-T, T], \Spb )$ and satisfies~\eqref{4} for all $t \in [-T, T] $. Since $T>0$ is arbitrary, this completes the proof. 
\end{proof}

\begin{lemma} \label{eSTE1} 
 Assume~\eqref{fDInt1}-\eqref{fSpa1b1} with $\alpha=1$.  
It follows that there exists a constant $C$ such that 
\begin{equation} \label{feSTE1:1} 
\sup  _{ 2m-1 \le  |\beta |\le 2m } \| \langle \cdot \rangle ^n D^\beta u \| _{ L^\infty  } \le C \Vert u \Vert_{ \Spb }
\end{equation} 
for all $u\in \Spb $.
\end{lemma}

\begin{proof} 
By density of $ \Srn $ in $ \Spb $, the result
follows if we prove estimate~\eqref{feSTE1:1} for $u \in \Srn $. Let $u \in \Srn $ and $ 2m-1\le  |\beta |\le 2m$. 
Since $H^k (\R^N ) \hookrightarrow L^\infty  (\R^N ) $ by~\eqref{fDInt1}, we have
\begin{equation*} 
\| \langle \cdot \rangle ^n D^\beta u \| _{ L^\infty  } \le C  \| \langle \cdot \rangle ^n D^\beta u \| _{ H^k } 
\le C \sum_{  |\rho |\le k } \| D^\rho ( \langle \cdot \rangle ^n D^\beta u ) \| _{ L^2 } .
\end{equation*} 
Moreover (see e.g.~\cite[Lemma A.1]{CazNau}), 
\begin{equation*}
\vert D^{\rho }(\langle x\rangle^{\eta} D^\beta u) \vert \leq C\sum
_{j=0}^{|\rho |}\langle x\rangle^{ n -|\rho |+j}\sum_{|\rho '| = j}|D^{\rho ' +\beta }u|  
\leq C \langle x\rangle^n \sum_{|\rho '| =  |\beta | } ^{ |\rho | +  |\beta |} |D^{\rho ' }u|  .
\end{equation*}
It follows that
\begin{equation*} 
\| \langle \cdot \rangle ^n D^\beta u \| _{ L^\infty  }  \le C   \sum_{|\rho | = 2m-1 } ^{2m +k } \| \langle x\rangle^n D^{\rho }u \| _{ L^2 }.
\end{equation*} 
Hence~\eqref{feSTE1:1} holds.
\end{proof}  

\begin{proof}
[Proof of Proposition~$\ref{eLE3}$]Since $\overline{e^{ it\gamma \langle
i\nabla t\rangle }  \psi }  = e^{ \mns it\gamma \langle i\nabla \rangle
} \overline{\psi}$ and the map $\psi\mapsto\overline{\psi}$ is isometric $ \Spa  \rightarrow \Spa $, we can restrict ourselves to the case $t\geq0$.

We let $\psi\in \Spa $. As before, we set $v(t)= e^{ \mns it\gamma\left\langle
i\nabla\right\rangle }\psi$. We begin by proving that if $t_{0}>0$ is given by Lemma~\ref{eLE1}, then $v(t)\in \Spa $ for all $0\leq t\leq t_{0}$ and
\eqref{eLE3:11} holds. By the definition \eqref{fSpa2} of the norm in the
space $ \Spa $, the estimate~\eqref{feSTE1:1}, and the embedding~\eqref{fInjYX}, we have
\[
\sum_{{  |\beta | }=0}^{2{m}} \Vert\langle x\rangle^{{n}}D^{\beta}%
\psi\Vert_{L^{\infty}}\leq C\Vert\psi\Vert_{ \Spa }.
\]
Then, using~\eqref{Low5b1} we have%
\begin{align*}
  \sum_{{j}=0}^{2{m}} \sup_{|\beta|={j}}\Vert\langle x\rangle^{{n}}D^{\beta
}v &\Vert_{L^{\infty}((0,t)\times \R^N )}\leq C(1+t)^{2m } 
\Vert\psi\Vert_{ \Spa } \\
&  +Ct(1+t)^{2m }\sup_{|\beta|=2{m}}\Vert\langle x\rangle^{{n}}D^{\beta
}v\Vert_{L^{\infty}((0,t)\times \R^N )}, 
\end{align*}
for all $0\leq t\leq t_{0}$. We estimate the last term in the above inequality
by~\eqref{feSTE1:1} and~\eqref{eLE3:11:0}. For $ |\beta |=2m$ we have
\begin{equation*} 
\begin{split} 
\Vert\langle x\rangle^{{n}}D^{\beta
}v\Vert_{L^{\infty}((0,t)\times \R^N )} & = \sup  _{ 0<s<t } \Vert\langle x\rangle^{{n}}D^{\beta
}v (s) \Vert_{L^{\infty} } \\ & \le C \sup  _{ 0<s<t } \Vert v (s) \Vert_\Spb 
\\ & \le C (1+t)^{{n}}\Vert\psi\Vert_{ \Spb } .
\end{split} 
\end{equation*} 
Using~\eqref{fInjYX}, we deduce that
\begin{equation} \label{5} 
 \sum_{{j}=0}^{2{m}}\sup_{|\beta|={j}}\Vert\langle x\rangle^{{n}}D^{\beta
}v\Vert_{L^{\infty}((0,t)\times \R^N )}\leq C(1+t)^{2m +n+1} 
\Vert\psi\Vert_{ \Spa } .
\end{equation} 
Using again~\eqref{eLE3:11:0}, we conclude that  $v(t)\in \Spa $ for all $0\le t\le t_0$ and that
 estimate~\eqref{eLE3:11} holds. 

Let us prove now \eqref{eLE3:12}. We consider a multi-index $\beta$ with
$0\leq|\beta|\leq2{m}$. Using \eqref{fInt1} and \eqref{Besselestimate} we
estimate%
\begin{align*}
\Vert\langle x\rangle^{n}D^{\beta}(v(t)-\psi)\Vert_{L^{\infty}}  &  \leq
\sum_{j=0}^{2m}\int_{0}^{t}\sup_{|\beta|=j}\Vert\langle x\rangle^{n}\langle
i\nabla\rangle^{-1}D^{\beta}v(s)\Vert_{L^{\infty}}ds\\
&  \leq t\sum_{j=0}^{2m}\sup_{|\beta|=j}\Vert\langle x\rangle^{n}D^{\beta
}v\Vert_{L^{\infty}((0,t)\times \R^N )}.
\end{align*}
Using \eqref{5} we conclude that~\eqref{eLE3:12} holds.

By Lemma~\ref{eLE3:0},  $v\in C([0,\infty), \Spb )$. Moreover, by
\eqref{eLE3:12} $v$ is continuous at $t=0$ in weighted $L^{\infty}$ norms.
Thus, $v$ is continuous at $t=0$ in $ \Spa $ norm. By the semigroup property we
conclude that $v\in C([0, \infty ), \Spa )$. This completes the proof.
\end{proof}

\begin{remark}
In particular, Proposition~\ref{eLE3} and \eqref{Low4} show that for $\psi
\in \Spa $ satisfying $\inf_{x\in \R^N }\langle x\rangle
^{n}|\psi(x)|>0$, the estimate from below $\inf_{x\in \R^N }\langle
x\rangle^{n}| e^{ \mns it\gamma\left\langle i\nabla\right\rangle }\psi(x)|>0$ holds,
for all $|t|$ sufficiently small. We do not know if this small time
requirement is necessary.
\end{remark}

We now prove estimates for the linear Dirac group $(e^{-itH})_{t\in \R }$ on the space
\begin{equation} \label{fDfnSpatilde} 
\SpaD = \Spa_ \CDiracd 
\end{equation} 
similar to the ones established in Proposition~\ref{eLE3}  for the group $(e^{-it\gamma \langle i\nabla
 \rangle })_{t\in \R }$. The free Dirac operator $H$
defined by~\eqref{HDirac}  with domain $D (  H )
= H^1 (\R^N , \C^\CDiracd) $ is a self-adjoint operator on $L^{2} ( \R^N , \C^{ \CDiracd} )  $, see~\cite{Thaller}.
Then, $(e^{itH})_{i \in \R}$ is a group of isometries on
$L^{2} ( \R^N , \C^{ \CDiracd} )  $, and on
$H^s (\R^N , \C^\CDiracd) $ for all
$s\in \R$. We have the following.

\begin{proposition} \label{eLE3Dirac}
Assume \eqref{fDInt1}-\eqref{fSpa1b1} with $\alpha=1$, and
let the space $ \SpaD $ be defined by
\eqref{fSpa1}-\eqref{fSpa2}, \eqref{fDfnell} and~\eqref{fDfnSpatilde}.
It follows that $e^{- i t H} \Psi \in C (\R, \SpaD)$ for all $\Psi \in  \SpaD$.
Moreover, there exist $C>0$  and $t_{0}>0$ such that
\begin{equation*} 
\Vert e^{-itH} \Psi \Vert_\SpaD  \leq C(1+|t|)^{2m+n+1}\Vert  \Psi \Vert_ \SpaD , 
\end{equation*} 
and
\begin{equation*} 
\sup_{|\beta|\leq2{m}} \Vert\langle x\rangle^{n}D^{\beta}(e^{-itH} \Psi 
-\psi)\Vert_{L^{\infty}}\leq C|t|(1+|t|)^{2m+n+1}\Vert \Psi \Vert
_\SpaD ,
\end{equation*} 
for all $\left\vert t\right\vert \leq t_{0}$ and all $ \Psi \in \SpaD $.
\end{proposition}

\begin{proof}
The proof is similar to the proof of Proposition~\ref{eLE3}. We only point out
the differences. To prove an estimate similar to Lemma~\ref{eLE1}, for a given
$0\leq{k}\leq{m-1}$ we apply Taylor's formula with integral remainder
involving the derivative of order $2 (  {m}-{k} )  $ to the function
$v=e^{-itH} \Psi $. Using the commutation relations \eqref{anti-commutation} we
see that $H^{2}=-\Delta+I $. It follows that $v(t) = e^{-itH} \Psi $ satisfies
\begin{equation*}
\begin{split} 
v(t)    = & \sum_{{j}=0}^{2{m}-2{k}-1}\frac{(it\gamma)^{{j}}}{{j}!}H^{{j}} \Psi \\
&  +\frac{\left(  i\gamma\right)  ^{2{m}-2{k}}}{(2{m}-2{k-1})!}\int_{0}%
^{t}(t-s)^{2{m}-2{k}-1}\left(  1-\Delta\right)  ^{ m- {k}}v(s)\,ds.
\end{split} 
\end{equation*}
Then, similarly to (\ref{kg1}), we deduce that
\begin{equation*}
\begin{split} 
\Vert\langle x\rangle^{{n}}D^{\beta}v\Vert_{L^{\infty}((0,t)\times{\mathbb{R}%
}^{N})}    \leq & C(1+t)^{2{m}}\sum_{{j}=0}^{{2m-2k-1}}\Vert\langle
x\rangle^{{n}}H^{j}D^{\beta} \Psi \Vert_{L^{\infty} }\\
&  +Ct(1+t)^{2{m}}\sum_{j={0}}^{{m-k}}\Vert\langle x\rangle^{{n}}D^{\beta
}\Delta^{j}v\Vert_{L^{\infty}((0,t)\times{\mathbb{R}}^{N})} ,
\end{split} 
\end{equation*}
where the $L^\infty $ norms are for $C^\CDiracd$-valued functions.
Since the matrices $\gamma_{j}$ and $\eta$ are unitary, we have
\begin{equation*} 
\Vert\langle x\rangle^{{n}}H^{j}D^{\beta} \Psi \Vert  _{ L^\infty } \leq C \sum_{|\rho|\leq2m}\Vert\langle x\rangle^{{n}}D^{\rho} \Psi  \Vert  _{ L^\infty } .
\end{equation*} 
Then, continuing as in the proof of Lemma \ref{eLE1} we obtain that
there exist $C>0$ and $t_{0}>0$ such that 
\begin{equation} \label{kg5}
\begin{split} 
 \sum _{{j}=0}^{2m-2}
 \sup & _{|\beta|={j}}\Vert\langle x\rangle^{{n}}D^{\beta}e^{-isH} \Psi 
\Vert_{L^{\infty}((0,t)\times \R^N  ) } \\
 \leq & C(1+t)^{2m} \sum _{|j|\leq2m} \Vert\langle x\rangle^{{n}}D^{\beta} \Psi \Vert_{L^{\infty} } \\ 
& +Ct(1+t)^{2m}\sup_{|\beta|=2{m}}\Vert\langle x\rangle^{{n}}D^{\beta}
e^{-isH} \Psi \Vert_{L^{\infty}((0,t)\times \R^N )},
\end{split} 
\end{equation} 
for all $0\leq t\leq t_{0}$ and all $\psi\in H^{{J}}({\mathbb{R}}^{N})$.

Next, we consider the space $\SpbD$ defined similarly to $\Spb $ (i.e. by~\eqref{fDfnSpbD1}-\eqref{fDfnSpbD2}), but for $C^\CDiracd$-valued functions instead of $C^2$-valued functions.
We claim that there exists a constant $C$ such that
\begin{equation} \label{kg4} 
\Vert e^{-itH} \Psi \Vert_ \SpbD \leq C(1+t)^{{n}}\Vert \Psi  \Vert _ \SpbD ,
\end{equation}
for all $t\geq0$ and all $\Psi \in \SpbD$. 
To see this, we consider $\Psi 
\in \Srn ^\CDiracd $ and we set $v(t)=e^{-itH} \Psi $, so that
$v\in C([0,\infty), \Srn ^\CDiracd )$.\ Let $k\in
\{1,\cdots,n\}$. Since $H$ is self-adjoint, similarly to~\eqref{3} we obtain
\begin{equation*} 
\frac{1}{2}\frac{d}{dt}\Vert\langle x\rangle^{k}D^{\beta}v \Vert_{L^{2}} 
^{2}= \Im \int_{ \R^N } (\langle x\rangle^{2k}HD^{\beta}v)\cdot
D^{\beta}\overline{v}.
\end{equation*} 
We now use the commutator relation 
\begin{equation*} 
[ \langle x\rangle^{2k},H]=2ki\langle x\rangle^{2k-2}\sum_{j=1}^{N}  \gamma_{j}x_{j}, 
\end{equation*} 
to prove similarly to~\eqref{kg3} that 
\begin{equation*} 
\frac{1}{2}\frac{d}{dt}\Vert\langle\cdot\rangle^{k}D^{\beta}v\Vert_{L^{2}}%
^{2}\leq C\Vert\langle\cdot\rangle^{k}D^{\beta}v\Vert_{L^{2}}\Vert\langle
\cdot\rangle^{k-1}D^{\beta}v\Vert_{L^{2}}. 
\end{equation*} 
Following the proof of Lemma \ref{eLE3:0}, we deduce~\eqref{kg4}.

Finally, we can use~\eqref{kg5} and~\eqref{kg4} to complete the proof of Proposition \ref{eLE3Dirac}
like the proof of Proposition~\ref{eLE3}. 
\end{proof}

\section{The nonlinear estimates} \label{sNL}

With the linear estimates at our disposal, we now establish estimates of the nonlinear terms.

We first estimate $ \NL  (  u )  $ in the space  $ \Spa $.

\begin{proposition} \label{eNL1} 
Let $\alpha>0$ and assume~\eqref{fDInt1}-\eqref{fSpa1b1}. 
Let $\Spa$ be defined by \eqref{fSpa1}-\eqref{fSpa2} and~\eqref{fDfnSpa}, and let $\NL (u)$ be defined by~\eqref{fdefnnl}.  
For
every $\eta>0$ and $u\in \Spa $ satisfying
\begin{equation}
\eta\inf_{x\in \R^N }(\langle x\rangle^{n}|u(x)|)\geq1
\label{eNL1:2}%
\end{equation}
it follows that $ \NL  (  u )  \in \Spa $. Moreover,
there exists a constant $C$ such that
\begin{equation}
\Vert \NL  (  u ) \Vert_ \Spa \leq C(1+\eta\Vert
u\Vert_ \Spa )^{2{J}}\Vert u\Vert_ \Spa ^{\alpha+1}
\label{eNL1:1}%
\end{equation}
for all $\eta>0$ and $u\in \Spa $ satisfying~\eqref{eNL1:2}.
Furthermore,
\begin{align}
  \Vert \NL & (  u_{1} )  - \NL  (  u_{2} ) \Vert_ \Spa \nonumber\\
&  \leq C\left(  1+\eta(\Vert u_{1}\Vert_ \Spa +\Vert u_{2}%
\Vert_ \Spa )\right)  ^{2{J}+1}(\Vert u_{1}\Vert_ \Spa +\Vert
u_{2}\Vert_{ \Spa })^{\alpha}\Vert u_{1}-u_{2}\Vert_ \Spa 
\label{eNL1:2b2}%
\end{align}
for all $\eta>0$ and $u, u_{1},u_{2}\in \Spa $ satisfying~\eqref{eNL1:2}.
\end{proposition}

\begin{proof}
Without loss of generality, we assume $ \Cstd  =1$.
First of all by \eqref{Besselestimate} we have%
\begin{equation}
\Vert \NL  (  u )  \Vert_ \Spa \leq C \Vert\,|u|^{\alpha
}u\Vert_ \Spa  \label{8}%
\end{equation}
and%
\begin{equation}
\Vert \NL  (  u_{1} )  - \NL  (  u_{2} )
\Vert_ \Spa \leq C \Vert\,|u_{1}|^{\alpha}u_{1}-|u_{2}|^{\alpha}u_{2}%
\Vert_ \Spa . \label{9}%
\end{equation}
Therefore, it is suffices to show that
\begin{equation}
\Vert\,\,|u|^{\alpha}u\Vert_ \Spa \leq C(1+\eta\Vert u\Vert
_ \Spa )^{2{J}}\Vert u\Vert_ \Spa ^{\alpha+1} \label{6}%
\end{equation}
and
\begin{align}
&  \Vert|u_{1}|^{\alpha}u_{1}-|u_{2}|^{\alpha}u_{2}\Vert_{ \Spa 
}\nonumber\\
&  \leq C\left(  (1+\eta(\Vert u_{1}\Vert_ \Spa +\Vert u_{2}%
\Vert_ \Spa )\right)  ^{2{J}+1}(\Vert u_{1}\Vert_ \Spa +\Vert
u_{2}\Vert_{ \Spa })^{\alpha}\Vert u_{1}-u_{2}\Vert_ \Spa .
\label{7}%
\end{align}
First, we calculate $D^{\beta}(|u|^{\alpha}u)$ with $1\leq
|\beta|\leq{J}$. We have
\[
D^{\beta}(|u|^{\alpha}u)=\sum_{\gamma+\rho=\beta}c_{\gamma,\rho}D^{\gamma
}(|u|^{\alpha})D^{\rho}u,
\]
where $c_{\gamma,\rho}$ are given by Leibnitz's rule. We write $|u|^{\alpha
}=(u\overline{u})^{\frac{\alpha}{2}}.$ Thus, the development of $D^{\beta
}(|u|^{\alpha}u)$ contains on the one hand the term
\begin{equation} \label{eNL1:3}
A=  |u|^{\alpha}D^{\beta}u, 
\end{equation}
and on the other hand, terms of the form
\begin{equation}
B=|u|^{\alpha-2p}D^{\rho}u\prod_{j=1}^{p}D^{\gamma_{1,j}}uD^{\gamma_{2,j}%
}\overline{u} \label{eNL1:4}%
\end{equation}
where
\begin{equation*} 
\gamma+\rho=\beta,\quad1\leq p\leq|\gamma|,\quad|\gamma_{1,j}+\gamma
_{2,j}|\geq1,\quad\sum_{j=0}^{p}(\gamma_{1,j}+\gamma_{2,j})=\gamma.
\end{equation*} 
First, we prove \eqref{6}. There are two possibilities. If $|\beta|\leq2{m} - 2 $,
we need to estimate the terms $\langle x\rangle^{n}A$ and $\langle
x\rangle^{n}B$ in $L^{\infty}.$ On the other hand, if $2{m} -1\leq|\beta|\leq
J$, we need to control the terms $\langle x\rangle^{n}A$ and $\langle
x\rangle^{n}B$ in $L^{2}$. Observe that the terms corresponding to $A$
contribute by $\Vert u\Vert_{L^\infty }^{\alpha}\Vert u\Vert_ \Spa .
$ Thus, we see that these terms are controlled by $(1+\eta\Vert
u\Vert_ \Spa )^{2{J}}\Vert u\Vert_ \Spa ^{\alpha+1}.$ Let us
focus on the terms~corresponding to $B.$ Using the lower bound \eqref{eNL1:2}
we have
\[
|u|^{\alpha-2p}\leq\eta^{2p}\langle x\rangle^{2p{n}}|u|^{\alpha}\leq\eta
^{2p}\langle x\rangle^{(2p-\alpha){n}}\Vert u\Vert_ \Spa ^{\alpha},
\]
hence
\begin{equation}
\left\vert B\right\vert \leq\eta^{2p}\langle x\rangle^{(2p-\alpha){n}}\Vert
u\Vert_ \Spa ^{\alpha}|D^{\rho}u|\prod_{j=1}^{p}|D^{\gamma_{1,j}%
}u|\,|D^{\gamma_{2,j}}u|. \label{eNL1:6}%
\end{equation}
We now consider separately the cases $|\beta|\leq2{m-2}$ and $2{m}-1 \leq|\beta|\leq
J.$

\noindent  {\bf The case}  $|\beta|\leq2{m-2}$.  We need to
estimate $\Vert\langle x\rangle^{{n}}B\Vert_{L^{\infty}}$. Since $|\beta
|\leq2{m-2 }$, all the derivatives in the right-hand side of \eqref{eNL1:6} are
also of order less than or equal to\ $2{m-2}$. Hence, all the derivatives in
\eqref{eNL1:6} are estimated by $\langle x\rangle^{-{n}}\Vert u\Vert
_ \Spa $; and so
\begin{equation} \label{eNL1:7}
\left\vert B\right\vert \leq(\eta\Vert u\Vert_ \Spa )^{2p}\langle
x\rangle^{-(\alpha+1){n}}\Vert u\Vert_ \Spa ^{\alpha+1}, 
\end{equation}
so that
\begin{equation*} 
\Vert\langle x\rangle^{{n}}B\Vert_{L^{\infty}}    \leq(\eta\Vert
u\Vert_ \Spa )^{2p}\Vert u\Vert_ \Spa ^{\alpha+1}\\
  \leq C(1+\eta\Vert u\Vert_ \Spa )^{2{J}}\Vert u\Vert_{ \Spa %
}^{\alpha+1}.
\end{equation*}

\medskip

\noindent  {\bf The case}   $2{m}-1\leq\left\vert \beta\right\vert \leq J$. We
need to estimate $\Vert\langle x\rangle^{{n}}B\Vert_{L^{2}}$. Suppose that one
of the derivatives in the right-hand side of~\eqref{eNL1:6} is of order
$\geq2{m} -1$, for instance $|\gamma_{1,1}|\geq2{m}-1$. The sum of the orders of
all derivatives in \eqref{eNL1:6} is equal to $|\beta|$. On the other hand, by
\eqref{fDInt1}, ${k}+{n} \leq2{m} -3 $ and $n\ge 2$, which implies that $\left\vert
\beta\right\vert \leq J=2{m}+2+{k \le }2{m}+{k}+{n} \leq4m -3$. Thus, we conclude
that all other derivatives in \eqref{eNL1:6} must have order $\leq2{m-2}$.
Hence, they are controlled by $\langle x\rangle^{-{n}}\Vert u\Vert
_ \Spa $. Therefore, from \eqref{eNL1:6} we get%
\begin{equation}
\left\vert B\right\vert \leq(\eta\Vert u\Vert_ \Spa )^{2p}\langle
x\rangle^{-\alpha{n}}\Vert u\Vert_ \Spa ^{\alpha}|D^{\gamma_{1,1}}u|.
\label{eNL1:9}%
\end{equation}
Since $2{m}-1\leq\left\vert \gamma_{1,1}\right\vert \leq J$, we estimate
$\Vert\langle x\rangle^{{n}}D^{\gamma_{1,1}}u\Vert_{L^{2}}\leq\Vert
u\Vert_ \Spa $. Hence, from \eqref{eNL1:9} we deduce%
\begin{equation*} 
\Vert\langle x\rangle^{{n}}B\Vert_{L^{2}}    \leq(\eta\Vert u\Vert
_ \Spa )^{2p}\langle x\rangle^{-\alpha{n}}\Vert u\Vert_{ \Spa %
}^{\alpha+1}
  \leq C(1+\eta\Vert u\Vert_ \Spa )^{2{J}}\Vert u\Vert_{ \Spa %
}^{\alpha+1}.
\end{equation*} 
Next, suppose that all the derivatives in the right-hand side
of~\eqref{eNL1:6} are of order $\leq2{m-2}$. In this case, we obtain
\eqref{eNL1:7} again. Multiplying~\eqref{eNL1:7} by $\langle
x\rangle^{{n}}$ we get
\begin{equation}
\langle x\rangle^{{n}}|B|\leq(\eta\Vert u\Vert_ \Spa )^{2p}\langle
x\rangle^{-\alpha{n}}\Vert u\Vert_ \Spa ^{\alpha+1}. \label{eNL1:10}%
\end{equation}
We need to estimate the $L^{2}$ norm of the last inequality. By the definition
of $n$ in \eqref{fDInt1} we have $\alpha{n}>\frac{N}{2}.$ Then, $\langle
x\rangle^{-\alpha{n}}\in L^{2}( \R^N ).$ Thus, it follows from
\eqref{eNL1:10} that%
\[
\left\Vert \langle x\rangle^{{n}}|B|\right\Vert _{L^{2}}\leq C(1+\eta\Vert
u\Vert_ \Spa )^{2{J}}\Vert u\Vert_ \Spa ^{\alpha+1}.
\]
Taking into account all the estimates, we obtain \eqref{6}; and using \eqref{8}, we
deduce \eqref{eNL1:1}.

Let us now prove \eqref{eNL1:2b2}. We develop both $D^{\beta}(|u_{1}|^{\alpha
}u_{1})$ and $D^{\beta}(|u_{2}|^{\alpha}u_{2})$ and use the
expressions~\eqref{eNL1:3} and~\eqref{eNL1:4} to expand the difference
$D^{\beta}(|u_{1}|^{\alpha}u_{1})-D^{\beta}(|u_{2}|^{\alpha}u_{2})$. \ On the
one hand, due to~\eqref{eNL1:3}, we get the term $|u_{1}|^{\alpha}D^{\beta
}u_{1}-|u_{2}|^{\alpha}D^{\beta}u_{2}$. We write this term as%
\begin{equation}
|u_{1}|^{\alpha}D^{\beta}u_{1}-|u_{2}|^{\alpha}D^{\beta}u_{2}=|u_{1}|^{\alpha
}(D^{\beta}u_{1}-D^{\beta}u_{2})+(|u_{1}|^{\alpha}-|u_{2}|^{\alpha})D^{\beta
}u_{2}. \label{10}%
\end{equation}
Similarly to the proof of \eqref{6}, we separate the cases $|\beta|\leq2{m} -2$
and $2{m}- 1\leq\left\vert \beta\right\vert \leq J$ and estimate the
$L^{\infty}$ and $L^{2}$ norms of \eqref{10}, respectively. We see that the
first term in the right-hand side of \eqref{10} can be controlled by $\Vert
u_{1}\Vert_{L^{\infty}}^{\alpha}\Vert u_{1}-u_{2}\Vert_ \Spa $ and
hence by the right-hand side of~\eqref{7}. In turn, the second term in the
right-hand side of \eqref{10} is estimated by $\Vert\,|u_{1}|^{\alpha}%
-|u_{2}|^{\alpha}\Vert_{L^{\infty}}\Vert u_{2}\Vert_ \Spa $. By
\eqref{eNL1:2}
\begin{align*}
|\,|u_{1}|^{\alpha}-|u_{2}|^{\alpha}|  &  \leq C(|u_{1}|^{-1}+|u_{2}%
|^{-1})(|u_{1}|+|u_{2}|)^{\alpha}|u_{1}-u_{2}|\\
&  \leq C\eta\langle x\rangle^{n}(|u_{1}|+|u_{2}|)^{\alpha}|u_{1}-u_{2}|\\
&  \leq C\eta(\Vert u_{1}\Vert_ \Spa +\Vert u_{2}\Vert_{ \Spa %
})^{\alpha}\Vert u_{1}-u_{2}\Vert_ \Spa ,
\end{align*}
and we obtain again a term which is controlled by the right-hand side of~\eqref{7}. 

Let us now estimate the terms that correspond to the difference of terms of the
form \eqref{eNL1:4} for $u_{1}$ and $u_{2}$. Each of these terms can be
written as
\begin{equation}
(|u_{1}|^{\alpha-2p}-|u_{2}|^{\alpha-2p})D^{\rho}u_{2}\prod_{j=1}^{p}%
D^{\gamma_{1,j}}u_{2}D^{\gamma_{2,j}}\overline{u_{2}} \label{fSPb2}%
\end{equation}
plus a sum of terms of the form
\begin{equation}
|u_{1}|^{\alpha-2p}D^{\rho}w\prod_{j=1}^{p}D^{\gamma_{1,j}}w_{1,j}%
D^{\gamma_{2,j}}\overline{w_{2,j}} \label{fSPb3}%
\end{equation}
where $w$, $w_{1,j}$, $w_{2,j}$ are all equal to either $u_{1}$ or $u_{2}$,
except one of them which is equal to $u_{1}-u_{2}$. The terms of the form
\eqref{fSPb3} are controlled by the right-hand side of~\eqref{7}, by using
\eqref{eNL1:2}. To estimate the term~\eqref{fSPb2} we use that
\begin{align*}
|\,|u_{1}|^{\alpha-2p}-|u_{2}|^{\alpha-2p}|  &  \leq C(|u_{1}|^{-2p-1}%
+|u_{2}|^{-2p-1})(|u_{1}|+|u_{2}|)^{\alpha}|u_{1}-u_{2}|\\
&  \leq C\eta^{2p+1}\langle x\rangle^{(2p+1){n}}(|u_{1}|+|u_{2}|)^{\alpha
}|u_{1}-u_{2}|\\
&  \leq C\eta^{2p+1}\langle x\rangle^{(2p-\alpha){n}}(\Vert u_{1}%
\Vert_ \Spa +\Vert u_{2}\Vert_{ \Spa })^{\alpha}\Vert
u_{1}-u_{2}\Vert_ \Spa .
\end{align*}
Then, proceeding as in the proof of \eqref{6}, we control \eqref{fSPb3} by the
right-hand side of~\eqref{7}. Thus, we see that \eqref{7} hold. Using
\eqref{9} we get \eqref{eNL1:2b2}. This completes the proof of Proposition
\ref{eNL1}.
\end{proof}

Now, we estimate 
\begin{equation} \label{fDfnNLD} 
\NL _1 (\Psi )  = \Cstt  | \Psi | ^{\alpha}  \Psi 
\end{equation} 
in the space $ \SpaD $.

\begin{proposition} \label{DiracNonlinear}
Let $\alpha>0$ and assume~\eqref{fDInt1}-\eqref{fSpa1b1}. 
Let  $ \SpaD $ be defined by
\eqref{fSpa1}-\eqref{fSpa2}, \eqref{fDfnell} and~\eqref{fDfnSpatilde}, and let $\NL_1 (\Psi )$ be defined by~\eqref{fDfnNLD}.  
For every $\eta>0$ and $\Psi \in \SpaD $ satisfying
\begin{equation} \label{kg6}
\eta\inf_{x\in \R^N } \langle x\rangle^{n}|\Psi(x)| \geq1
\end{equation}
it follows that $\NL _1 ( \Psi ) \in \SpaD $. Moreover, there exists a constant $C>0$ such that
\begin{equation*} 
 \| \NL _1 (\Psi ) \|_\SpaD \le C (1 + \eta  \| \Psi \|_\SpaD )^{2J} \| \Psi \|_\SpaD ^{\alpha +1} ,
\end{equation*} 
and 
\begin{equation*} 
\begin{split} 
 \| \NL _1 (\Psi _1) - & \NL _1 (\Psi _2) \|_\SpaD \\ &
  \le C (1 + \eta  ( \| \Psi _1 \|_\SpaD +  \| \Psi _2 \|_\SpaD ) )^{2J +1} (   \| \Psi _1 \|_\SpaD +  \| \Psi _2 \|_\SpaD  )^{\alpha }   \| \Psi _1 - \Psi _2\|_\SpaD ,
\end{split} 
\end{equation*} 
for all $\eta>0$ and $ \Psi , \Psi_{1},\Psi_{2}\in \SpaD$ satisfying~\eqref{kg6}.
\end{proposition}

\begin{proof}
The proof of Proposition \ref{eNL1} uses only formulas~\eqref{8} and~\eqref{9}.
Since $\NL _1$ clearly satisfy these, the result follows. 
\end{proof}

\section{Proofs of Theorems~\ref{eThm1b1} and \ref{ThDirac} \label{sNLS}}

We are now in position to prove our main results. Theorem~\ref{eThm1b1} will be
consequence of the following existence result for the Cauchy problem
\begin{equation} \label{11}
\begin{cases} 
i\partial_{t}u- \gamma \left\langle i\nabla\right\rangle u=  \LINL  ( u )  , \\
u (  0 )  =u_{0},
\end{cases} 
\end{equation}
where
\begin{equation*} 
\LINL (u)= \LI (u) + \NL (u) ,
\end{equation*} 
which we study in the equivalent form (Duhamel's formula)
\begin{equation} \label{fIE1} 
u(t)= e^{ \mns i t \gamma \langle i\nabla \rangle } u_0 - i\int_{0}^{t} e^{ \mns i(t-s) \gamma \langle i\nabla \rangle } 
\LINL \left(  u\right)  \,ds .
\end{equation} 

\begin{proposition}
\label{eNLS1} Let $\alpha>0$ and $  \Cstd  \in{ \C }$.
Assume~\eqref{fDInt1}-\eqref{fSpa1b1} and let the space $ \Spa = \Spa _2$
 be defined by \eqref{fSpa1}-\eqref{fSpa2} and~\eqref{fDfnSpa}. If $u_{0}\in \Spa $
satisfies
\begin{equation}  \label{eThm1:1}
\inf_{x\in \R^N }\langle x\rangle^{n}|u_{0}(x)|>0 ,
\end{equation}
then there exist $T > 0$ and a unique solution $u\in C([-T,T], \Spa )$ of~\eqref{11}.
Moreover,
\begin{equation} \label{eThm1:2}
\inf _{ -T\le t\le T } \inf_{x\in \R^N }\langle x\rangle^{n}|u (t, x)| > 0 .
\end{equation} 
\end{proposition}

\begin{proof}
We first prove uniqueness. Suppose $T>0$ and $u_1, u_2 \in C([-T, T], \Spa ) $ are two solutions of~\eqref{fIE1}.  
Using~\eqref{fdefnnl}, \eqref{Besselestimate} (with $n=0$), and $\Spa \hookrightarrow H^J (\R^N ) \hookrightarrow L^\infty  (\R^N ) $ we see that
\begin{equation*} 
 \| \NL (u_1) - \NL (u_2)  \| _{ L^2 } \le C  \| \,  | \mathbf{a} \cdot u_1 |^\alpha  \mathbf{a} \cdot u_1 - | \mathbf{a} \cdot u_2 |^\alpha  \mathbf{a} \cdot u  _2 \| _{ L^2 } \le C  \| u_1 - u_2 \| _{ L^2 },
\end{equation*} 
and
\begin{equation*} 
 \| \LI (u ) \| _{ L^2 } \le C   \| u  \| _{ L^2 }. 
\end{equation*} 
Since $(  e^{ \mns i t \gamma \langle i\nabla \rangle } )  _{ t\in \R }$ is a group of isometries on $L^2 (\R^N ) $, we deduce that
\begin{equation*} 
 \| u_1 (t)- u_2 (t) \| _{ L^2 } \le C  \Bigl| \int _0^t  \| u_1 (s)- u_2 (s) \| _{ L^2 } ds  \Bigr|,
\end{equation*} 
and uniqueness follows by Gronwall's inequality.

Next, we use the linear estimates of Proposition~\ref{eLE3} and the nonlinear
estimates of Proposition~\ref{eNL1} to prove the local existence result by a
contraction mapping argument. We let
\[
\eta>0,\quad K>0,\quad0<T \leq t_{0},
\]
where $t_{0}$ is given by Proposition \ref{eLE3}. We define the set $ \Ens $ by
\[%
\begin{split} 
\Ens =\{u\in C([-T,T], \Spa );  &  \,\Vert u\Vert_{L^{\infty
}( (-T,T), \Spa )}\leq K \\ & \text{ and }
\eta \inf  _{ x\in \R^N  }\langle x\rangle^{n}|u(t,x)|\geq1    \text{ for }-T<t<T  \} ,
\end{split} 
\]
so that $\Ens $ equipped with the distance $ \dist (u,v)=\Vert u-v\Vert
_{L^{\infty}((-T,T), \Spa )}$ is a complete metric space. For given
$u\in \Ens $ and $u_{0}\in \Spa $, we set%
\[
\Phi_{u}(t)= - i\int_{0}^{t} e^{ \mns i(t-s) \gamma  \langle i\nabla \rangle
} \LINL  (  u (s) )  \,ds
\]
and
\[
\Psi_{ u_{0},u}(t)= e^{ \mns it\gamma\left\langle i\nabla\right\rangle }u_{0}+\Phi
_{u}(t)
\]
for $-T<t<T$. By the definition of $\Ens $ and Proposition~\ref{eNL1}
we see that if $u\in \Ens $, then $ \NL \left(  u\right)  \in
C([-T,T], \Ens )$ and
\begin{equation}
\Vert  \NL  (  u )  \Vert_{L^{\infty}( (-T,T), \Spa 
)}\leq C(1+\eta K)^{2{J}}K^{\alpha+1}. \label{fSp1b}
\end{equation}
moreover, by~\eqref{Besselestimate}, 
\begin{equation}  \label{fSp1d}
\LI  \in  {\mathcal L} ( \Spa ), 
\end{equation} 
so that
\begin{equation}
\Vert  \LI  (  u )  \Vert_{L^{\infty}( (-T,T), \Spa 
)}\leq CK . \label{fSp1c}
\end{equation}
In addition, it follows from Proposition~\ref{eLE3} and the semigroup property that $\Phi_{u}\in
C([-T,T], \Spa )$. From~\eqref{eLE3:11}, \eqref{fSp1b} and~\eqref{fSp1c}
we estimate
\begin{equation}
\Vert\Phi_{ u }\Vert_{L^{\infty}( (-T,T), \Spa )}\leq C T K_{t_{0}} [ K +  
  (1+\eta K)^{2{J}}K^{\alpha+1} ] \label{fNLS4}%
\end{equation}
and
\begin{equation}
\Vert\Psi_{ u_{0}, u }\Vert_{L^{\infty}( (-T,T), \Spa )}\leq CK_{t_{0}%
}\left(  \Vert u_{0}\Vert_ \Spa + TK +  T(1+\eta K)^{2{J}}%
K^{\alpha+1}\right)  , \label{fNLS5}%
\end{equation}
where 
\[
K_{t_{0}}=(1+ t_{0} )^{ 2m + n +1}.
\]
Arguing similarly and using~\eqref{eNL1:2b2}, we estimate
\begin{equation}
\Vert\Phi_{v}-\Phi_{w}\Vert_{L^{\infty}( (-T,T), \Spa )}\leq C T K_{t_{0}%
} [1 + (1+\eta K)^{2{J}+1}K^{\alpha} ] \dist (v,w), \label{fNLS6}%
\end{equation}
for all $v,w\in \Ens $. Next, using~\eqref{eLE3:12}, \eqref{fNLS4} and the inequality $ \langle \cdot \rangle ^n  |u (\cdot )| \le  \| u \|_\Spa $, we see that
\begin{equation} \label{fNLS7}
\begin{split} 
\langle x\rangle^{n}|\Psi_{u_{0},u}(t,x)| & \ge \inf _{ x\in \R^N  } \langle x\rangle^{n}  | u_0 (x)| - C T K _{ t_0 }  \|u_0 \|_\Spa -  \| \Phi _u \| _\Spa \\ 
 \ge \inf _{ x\in \R^N  } \langle x\rangle^{n} & | u_0 (x)| - C T K _{ t_0 } (  \|u_0 \|_\Spa + K + (1+\eta K)^{2{J}}K^{\alpha+1}) ;
\end{split} 
\end{equation} 
Having all the necessary estimates, we now argue as follows. Let $u_{0}%
\in \Spa $ be such that $\inf_{x\in \R^N }\langle
x\rangle^{n}|u_{0}(x)|>0$. We let
\begin{align}
\eta &  =2(\inf_{x\in \R^N }\langle x\rangle^{n}|u_{0}%
(x)|)^{-1}\label{fNLS8}\\
K  &  =2 \widetilde{C}  K_{t_{0}}\Vert u_{0}\Vert_ \Spa . \label{fNLS9}%
\end{align}
where $ \widetilde{C}  $ is the supremum of the constants $C$
in~\eqref{fNLS4}--\eqref{fNLS7}. In particular we see that $u(t)\equiv u_{0}$
belongs to $\Ens $, so that $ \Ens  \not =\emptyset$. We let $T\in (0, t_0]$
be sufficiently small so that
\begin{gather}
 \widetilde{C} T  K_{t_{0}} [ 1+  (1+\eta K)^{2{J}+1}K^{\alpha } ] \leq\frac{1}%
{2}\label{fNLS10}\\
 \widetilde{C} T K_{t_{0}} (  \Vert u_{0}\Vert_ \Spa + K + (1+\eta
K)^{2{J}}K^{\alpha+1} )  \leq\frac{1}{\eta}. \label{fNLS11}%
\end{gather}
Then, applying~\eqref{fNLS5}, \eqref{fNLS9} and~\eqref{fNLS10} we obtain
\[
\Vert\Psi_{u_{0},u}\Vert_{L^{\infty}( (-T,T), \Spa )}\leq K .
\]
Moreover, inequalities~\eqref{fNLS7}, \eqref{fNLS8}
and~\eqref{fNLS11} imply that
\begin{equation} \label{fEstin1} 
\eta \inf  _{ x\in \R^N  } \langle x\rangle^{n}|\Psi_{u_{0},u}(t,x)|\geq1
\end{equation} 
for $-T\leq t\leq T$. It follows that $\Psi_{u_{0},u}\in  \Ens $ for all $u\in  \Ens $.
Using \eqref{fNLS6} and~\eqref{fNLS10} we deduce that the map $u\mapsto
\Psi_{u_{0},u}$ is a strict contraction $ \Ens  \rightarrow  \Ens $. Therefore, it has a fixed point, which is a solution of~\eqref{fIE1}, and estimate~\eqref{eThm1:2}  follows from~\eqref{fEstin1}. 
This completes the proof.  
\end{proof}

\begin{proof}
[Proof of Theorem $\ref{eThm1b1}$]Consider the problem \eqref{KGCauchy}. Suppose
that the initial data are such that $w_{0}\in \Spa _{1}$, $ \langle
i\nabla \rangle ^{-1} w _{1}\in \Spa _{1}$ and \eqref{BoundBelow}
holds. It follows that $u_0$ defined by
\begin{equation} \label{fDfnuz} 
u_0 =\frac{1}{2} ( w_0  \mathbf{a} + i [ \langle i\nabla \rangle
^{-1} w_1]  \mathbf{b}  ) 
\end{equation} 
belongs to $\Spa$. 
Moreover, using~\eqref{fDfnab},
\begin{equation*} 
 | u_0|^2 = 2 (  |w_0|^2 +  |  \langle i\nabla \rangle
^{-1} w_1 |^2 )
\end{equation*} 
 so that $u_{0}$ satisfies \eqref{eThm1:1}. It follows from Proposition
\ref{eNLS1} that there exist $T>0$ and a solution
\begin{equation}
u\in C([-T,T], \Spa ) \label{12}%
\end{equation}
of \eqref{11}. 
Since $\Spa \hookrightarrow H^J (\R^N ) $ by~\eqref{fSpa1:b2} we have $ \gamma \langle i \nabla \rangle u\in C([-T, T], H^{J-1} (\R^N ) ) $.  Moreover, $\LINL (u)\in C ([-T, T], \Spa )$ by Proposition~\ref{eNL1} and~\eqref{fSp1d}. Equation~\eqref{11}, yields $u\in C^1 ( [-T, T], H^{J-1} (\R^N ) ) $, so that $ \gamma \langle i \nabla \rangle u\in C^1 ([-T, T], H^{J-2} (\R^N ) ) $.
In addition, one verifies easily that $\NL (u) \in C^1 ( [-T, T], H^{J-1} (\R^N ) ) $ and
\begin{equation*} 
\partial _t \NL (u) = - \frac { \Cstd  } {2}  \Bigl[ \langle i\nabla \rangle ^{-1} \Bigl(  \frac {\alpha +1} {2} | \mathbf{a} \cdot u |^\alpha  \mathbf{a} \cdot \partial _t u +  \frac {\alpha } {2} | \mathbf{a} \cdot u |^{\alpha -2} (\mathbf{a} \cdot u)^2 \mathbf{a} \cdot \partial _t  \overline{u}   \Bigr) \Bigr]  \mathbf{b} .
\end{equation*} 
Furthermore, $\LI  \in {\mathcal L} (H^{J-1} (\R^N ) )$ by~\eqref{Besselestimate} (with $n=0$ and $ p =2$), so that  $\LI (u) \in C^1 ( [-T, T], H^{J-1} (\R^N ) ) $.
Using again equation~\eqref{11}, we conclude that $u\in C^2 ([-T, T], H^{J-2} (\R^N ) )$. 

We now define
\begin{equation}  \label{13:b1}
w \in C ([-T, T], \Spa _1 ) \cap C^2 ([-T, T], H^{J-2} (\R^N ) )
\end{equation}  
by
\begin{equation}
w=\mathbf{a} \cdot u. \label{13}%
\end{equation}
Since $J-2 > \frac {N} {2}$, we have in particular $w \in  C^2 ( [-T, T] \times \R^N )$. 
Next, note that
\begin{equation}
\gamma ^2 = I,\quad 
\gamma \mathbf{a}=\mathbf{b}, \quad  \mathbf{a}
\cdot \mathbf{b}=0,\quad \gamma \mathbf{b} = \mathbf{a},\quad \mathbf{a} \cdot \mathbf{a}=2. \label{id}%
\end{equation}
Using equation~\eqref{11} and~\eqref{id} we see that 
\begin{equation*} 
\begin{split} 
w _{ tt }- \Delta w + w & =  \mathbf{a} \cdot  ( u  _{ tt }- \Delta u + u ) \\
& =  - \mathbf{a} \cdot  [ (i\partial _t + \gamma \langle i\nabla \rangle )  (i\partial _t u - \gamma \langle i\nabla \rangle u)  ]\\
& = -  \mathbf{a} \cdot  [ (i\partial _t + \gamma \langle i\nabla \rangle ) \LINL (u) ] 
 =  - \mathbf{a} \cdot  [  \gamma \langle i\nabla \rangle  \LINL (u) ]  \\
& = -  \frac { \Cstu -1  } {2} \mathbf{a} \cdot  [  \gamma  (   w \mathbf{b} ) ]   +  \frac { \Cstd  } {2} \mathbf{a} \cdot  [  \gamma  (  |w|^\alpha w \mathbf{b} ) ] \\
&  =  (1- \Cstu ) w+  \Cstd  |w|^\alpha w .
\end{split} 
\end{equation*} 
Moreover, by~\eqref{fDfnuz}  and~\eqref{id},
\[
w  (  0 )  =\mathbf{a} \cdot u _0
= \frac{1}{2} \mathbf{a} \cdot  ( w_0  \mathbf{a} + i [ \langle i\nabla \rangle
^{-1} w_1]  \mathbf{b}  ) 
= w_{0}.
\]
Similarly, using~\eqref{11}, \eqref{fDfnuz}  and~\eqref{id},
\begin{equation*} 
\begin{split} 
w_t (0) & = \mathbf{a} \cdot  \partial _t u(0) =-i  \mathbf{a} \cdot  ( \gamma \left\langle i\nabla\right\rangle u_0 + \LINL  ( u_0 ) ) \\
& = -i \mathbf{a} \cdot  ( \gamma \left\langle i\nabla\right\rangle u_0  )
=-  \frac {i} {2}  \mathbf{a} \cdot  ( \gamma  [  \langle i\nabla \rangle w_0 \mathbf{a} + i w_1 \mathbf{b}]  ) = w_1.
\end{split} 
\end{equation*} 
Thus we see that $v$ solves \eqref{KGCauchy}. 
This proves the existence part. 

Since $\Spa _1 \hookrightarrow L^\infty  (\R^N ) $, 
uniqueness easily follows from standard energy estimates. 

Finally, suppose that $w_0 $ satisfies \eqref{14}. Taking the scalar product of equation \eqref{11} with
$\mathbf{a} $, integrating in time and
using \eqref{id} and \eqref{13}, we obtain
\begin{equation*} 
 w (  t )  = w _{0}-i\int_{0}^{t} \langle i\nabla \rangle
\mathbf{b  } \cdot u.
\end{equation*} 
Then,%
\begin{equation*} 
\begin{split} 
\inf_{x\in \R^N }\langle x\rangle^{n} \vert w (  t )
 \vert & \geq\inf_{x\in \R^N }\langle x\rangle^{n} w_{0}%
-t\Vert\langle x\rangle^{n}\left\langle i\nabla\right\rangle  u\Vert_{L^{\infty}((0,t)\times \R^N )} \\
& = \inf_{x\in \R^N }\langle x\rangle^{n} w_{0}%
-t\Vert\langle x\rangle^{n}\left\langle i\nabla\right\rangle ^{-1}\left(
1- \Delta \right)  u\Vert_{L^{\infty}((0,t)\times \R^N )}  . 
\end{split} 
\end{equation*} 
By Lemma \ref{BE} 
\begin{equation*} 
\Vert\langle x\rangle^{n}\left\langle i\nabla\right\rangle
^{-1}\left(  1- \Delta \right)  u\Vert_{L^{\infty}((0,t)\times{ \R 
}^{N})}\leq C\Vert\langle x\rangle^{n}\left(  1- \Delta \right)
u\Vert_{L^{\infty}((0,t)\times \R^N )}.
\end{equation*} 
Using \eqref{12} we see that there is $0<T_{1}\le T $ such that \eqref{15} holds. 
\end{proof}

Finally, we complete the proof of Theorem \ref{ThDirac}.

\begin{proof} [Proof of Theorem~$\ref{ThDirac}$]
We use Duhamel's formula to reformulate equation~\eqref{Dirac} in the equivalent form
\begin{equation*} 
\Psi(t)=e^{-itH}\Psi_{0}-i\int_{0}^{t}e^{-i(t-s)H} \NL_1 ( \Psi (s)) \,ds,
\end{equation*} 
where $\NL_1 (\Psi )$ is given by~\eqref{fDfnNLD}.  
Theorem \ref{ThDirac} now follows from a standard contraction mapping argument 
(exactly as in the proof of Proposition \ref{eNLS1}) based on  the linear estimates of Proposition~\ref{eLE3Dirac} and the nonlinear estimates of Proposition~\ref{DiracNonlinear}.
\end{proof}

\end{document}